\documentclass[11pt,reqno]{amsart}

\usepackage{amsfonts,amsfonts,amssymb}
\usepackage[foot]{amsaddr}
\usepackage{cite}
\usepackage{color}
\usepackage{hyperref}
\usepackage{graphicx}
\usepackage{mathtools}
\usepackage[sc]{mathpazo}
\usepackage{geometry}
\allowdisplaybreaks

\theoremstyle{plain}
\newtheorem{theorem}{Theorem}[section]
\newtheorem*{theorem*}{Theorem}
\newtheorem{lemma}[theorem]{Lemma} 
\newtheorem{proposition}[theorem]{Proposition} 
 
\newtheorem{hypothesis}[theorem]{Hypothesis}
\theoremstyle{remark}
\newtheorem{remark}[theorem]{Remark}

\definecolor{Myblue}{cmyk}{1, 1, 0, 0}

\newcommand{\eps}{{\varepsilon}}
\newcommand{\al}{\alpha}
\newcommand{\be}{\beta}

\renewcommand{\d}{\mathrm{d}}
\newcommand{\cupl}{\bigcup\limits}
\newcommand{\suml}{\sum\limits}
\newcommand{\intl}{\int\limits}

\newcommand{\dist}{\mathrm{dist}}
\newcommand{\dom}{\mathrm{dom}}

\newcommand{\N}{\mathbb{N}}
\newcommand{\Z}{\mathbb{Z}}
\newcommand{\R}{\mathbb{R}}
\newcommand{\Zm}{\Z\setminus\N}
\newcommand{\C}{\mathcal{C}}

\renewcommand{\L}{\mathsf{L}^2} 
\newcommand{\W}{\mathsf{W}}
\renewcommand{\C}{\mathsf{C}}
\newcommand{\HS}{V}
\newcommand{\HSk}{\mathbf{V}}

\renewcommand{\H}{\mathcal{H}}  
\newcommand{\h}{\mathfrak{h}} 
\newcommand{\Hk}{\mathbf{H}} 
\newcommand{\hk}{\mathbf{h}} 

\renewcommand{\u}{\mathbf{u}}
\newcommand{\vv}{\mathbf{v}}
\newcommand{\w}{\mathbf{w}}

\newcommand{\ess}{\mathrm{ess}}
\newcommand{\disc}{\mathrm{disc}}

\renewcommand{\aa}{_{\alpha,\infty}}
\newcommand{\aab}{_{\alpha,\beta}}

\newcommand{\I}{\mathcal{I}}
\newcommand{\Id}{\mathrm{I}}

\newcommand{\FF}{\mathcal{F}}

\newcommand{\interval}{{(\ell_-,\ell_+)}}

\newcommand{\ds}{\displaystyle}

\newcommand{\opset}{{\mathcal{O}}}

\begin{document}

\title[Singular Schr\"odinger operators with prescribed spectral properties]{Singular Schr\"odinger operators with prescribed spectral properties}

\author{Jussi Behrndt\,$^1$}
\address{$^1$ Institute of Applied Mathematics, Graz University of Technology, Austria}
\email{behrndt@tugraz.at}

\author{Andrii Khrabustovskyi\,$^{2,3}$}
\address{$^2$ Department of Physics, Faculty of Science, University of Hradec Kr\'alov\'e, Czech Republic}
\address{$^3$ Department of Theoretical Physics,
Nuclear Physics Institute of the Czech Academy of Sciences, \v{R}e\v{z}, Czech Republic} 
\email{andrii.khrabustovskyi@uhk.cz}

\keywords{Schr\"{o}dinger operator, $\delta$-interaction, essential spectrum, discrete spectrum}

\clearpage\maketitle

\numberwithin{equation}{section}

\begin{abstract}
	The paper deals with singular Schr\"odinger operators of the form
	\begin{gather*}
    -{\mathrm{d}^2\over \mathrm{d} x^2 }  + \sum_{k\in\mathbb{Z} }\gamma_k \delta(\cdot-z_k),\quad 
    \gamma_k\in\mathbb{R},
	\end{gather*}
	in $\mathsf{L}^2(\ell_-,\ell_+)$, where $(\ell_-,\ell_+)$ is a bounded interval, and 
	$ \delta(\cdot-z_k)$ is the Dirac delta-function supported at $z_k\in (\ell_-,\ell_+)$.  
	It will be shown that the interaction strengths $\gamma_k$ and the points $z_k$ can be chosen in such a way that the essential spectrum and a bounded part of the discrete spectrum of this self-adjoint operator coincide with prescribed sets on a real line. 	
\end{abstract}

\section{Introduction\label{sec:intro}}

Self-adjoint Laplace and Schr\"odinger operators on bounded domains typically have purely discrete spectrum since in many situations the operator or corresponding form domain is compactly embedded in the underlying $L^2$-space.
In general, however, this is not true, and a well known example is the Neumann Laplacian on a bounded non-Lipschitz domain discussed by
R.~Hempel, L.~Seco, and B.~Simon in \cite{HSS91}. More precisely, for an arbitrary closed set $S\subset [0,\infty)$ 
a bounded domain $\Omega$ was constructed in \cite{HSS91} such that the essential spectrum of the 
Neumann Laplacian $-\Delta_\Omega^N$ on $\Omega$ coincides with the set $S$. In particular, in the case 
$0\in S$ one can use a domain $\Omega$ consisting of a series of ``rooms  and passages'', see Figure~\ref{fig0}. 
These results were further elaborated by R.~Hempel, T.~Kriecherbauer, and P.~Plankensteiner in \cite{HKP97}, where also 
a prescribed bounded part of the discrete spectrum was realized by constructing domain with a certain ``comb'' structure.
Furthermore, in \cite{S92} B.~Simon found a bounded domain of ``jelly roll'' form such that the spectrum of $-\Delta_\Omega^N$ is purely absolutely continuous and covers $[0,\infty)$. Note that in the above situations the peculiar spectral properties of $-\Delta^N_\Omega$  are all 
caused by irregularity of $\partial\Omega$.
Another approach to construct Laplace or Schr\"{o}dinger operators on bounded domains with non-standard spectral properties is to choose ``unusual``
boundary conditions; e.g. $0$ is always an eigenvalue of infinite multiplicity of the Krein-von Neumann realization of  $-\Delta$.
In the abstract setting S.~Albeverio, J.~Brasche, M.~Malamud, H.~Neidhardt, and J.~Weidmann  \cite{B04,BNW93.2,BNW93.1,BN96,BN95,BM99,Br89,ABN98,ABMN05} 
consider a symmetric operator $S$ in a Hilbert space with infinite deficiency indices such that $\sigma(S)$ has a gap, and discuss the possible spectral 
properties of self-adjoint extensions of $S$ in the gap; cf.\cite{ABN98} for applications to the Laplacian. In this context we also refer the reader to the 
recent expository paper \cite{BK21}.

\subsection{Setting of the problem and the main result}

In the present paper we consider (one-dimensional) Schr\"odinger operators with $\delta$-interactions defined by the formal expression
\begin{gather}
\label{delta-formal}
\H_\gamma=-{\d^2\over \d x^2 }  + \suml_{k\in K}\gamma_k \delta(\cdot-z_k);
\end{gather}
here $ \delta(\cdot-z_k)$ is the Dirac delta-function supported at $z_k$, $\gamma_k\in\R$, and $K$ is a countable set.
Such operators can be regarded as so-called \emph{solvable models} in quantum mechanics 
describing the motion of a particle in a
potential supported by a discrete (finite or infinite) set of
points; cf. the monograph \cite{AGHH05} for more details. 
Now assume that all points $z_k$ are contained in a (bounded or unbounded) interval $\interval$.
If the set $K$ is finite the formal expression \eqref{delta-formal} can be realized as a self-adjoint operator in $\L\interval$ 
with the action  
\begin{gather*}
-(u\restriction_{\interval\setminus \cupl_{k\in K}\{z_k\}})''   
\end{gather*}
defined for functions  $u\in \mathsf{H}^2(\interval\setminus \cupl_{k\in K}\{z_k\})$ satisfying 
\begin{gather}\label{delta-conditions}
u(z_k-0) = u(z_k + 0),\quad u'(z_k+0)-u'(z_k-0)=\gamma_k u( z_k\pm 0) 
\end{gather}
at the points  $z_k$
and suitable conditions at the endpoint of $\interval$ (e.g., $u(\ell_-)=u(\ell_+)=0$).
If $K$ is a countable infinite set then the definition of $\H_\gamma$ is more subtle, in particular, if $|z_k-z_{k-1}|\to 0$ as $|k|\to \infty$; 
cf. \cite{AKM10,KM10}. 
If $\interval$ is bounded and $K$ is finite then the spectrum of $\H_\gamma$ is purely discrete, but if 
$K$ is infinite then the essential spectrum of $\H_\gamma$ may be non-empty, even if $\interval$ bounded.

The goal of the present paper is to show that for an arbitrary closed semibounded (from below) set $S_\ess$ one can
construct an operator $\H_\gamma$ of the form \eqref{delta-formal} on a bounded interval  such that $\sigma_\ess(\H_\gamma)=S_\ess$ and, in addition, 
a bounded part of the discrete spectrum can be controlled.
More precisely, assume that we have a {set} $S_{\ess}\subset\R$, a {sequence} of real 
numbers  $S_{\disc}=(s_k )_{k\in\N}$ and a 
bounded interval $(T_1,T_2)\subset\R$ such that
\begin{align}
\label{Sprop1}
& S_{\ess}\text{ is closed and bounded from below},
\\
\label{Sprop2}
& S_{\ess}\cap [T_1,T_2]=\overline{\opset}\text{, where }\opset\subset (T_1,T_2)\text{ is an open set},
\\
\label{Sprop3}
& s_k\in (T_1,T_2)\setminus \overline{\opset},\ \forall k\in\N,
\\
\label{Sprop4}
& s_k\not= s_l\text{ as }k\not=l,
\\
\label{Sprop5}
& \text{all accumulation points of }S_\disc\text{ are contained in }S_\ess.
\end{align}
From \eqref{Sprop1}--\eqref{Sprop5} we conclude that
\begin{gather}\label{Sprop:add1}
\text{each }s_k\text{ has a punctured neighborhood containing no other points of }
S_{\ess}\cup S_{\disc}.
\end{gather}
Also note that in view of \eqref{Sprop2} $S_{\ess}$ has no isolated points in
$[T_1,T_2]$.
\smallskip 

The following theorem is the main result of this paper, see also Theorem~\ref{th:main:pre} for a slightly more rigorous formulation with the formal operator
$\H_\gamma$ replaced by the precisely defined operator $\H\aab$ from Section~\ref{sec:co}.

\begin{theorem}\label{th:main}
There exists a bounded interval $\interval\subset\mathbb{R}$, a sequence of points $(z_k)_{k\in\Z}$ with $z_k\in \interval$ and
a sequence of real numbers $(\gamma_k)_{k\in\Z}$ such that the operator $\H_\gamma$ in $\L\interval$ defined by the formal expression \eqref{delta-formal} satisfies
\begin{gather}\label{main}
	\sigma_\ess(\H_\gamma)=S_\ess,\quad	\sigma_\disc(\H_\gamma)\cap(T_1,T_2)= S_\disc,
\end{gather}
moreover 
\begin{gather}\label{main+}
\text{all eigenvalues at }\sigma_\disc(\H_\gamma)\cap(T_1,T_2)\text{ are simple.}
\end{gather}
\end{theorem}

We mention that besides \cite{HSS91,HKP97} our research is also inspired by a celebrated paper of Y.~Colin de Verdi\'{e}re \cite{CdV87}, where 
a Riemannian metric $g$ on a given compact manifold $M$ is constructed such that the first $m$ eigenvalues of 
the Laplace-Beltrami operator on $(M,g)$ coincide with prescribed numbers; similar results were also obtained for the Neumann Laplacian 
and regular Schr\"odinger operators.

\subsection{Sketch of the proof strategy}

To construct the operator $\H_\gamma$
satisfying \eqref{main}--\eqref{main+} we utilize ideas of the  aforementioned paper \cite{HSS91}, where 
a bounded domain $\Omega$ was constructed such that the essential spectrum of the Neumann Laplacian $-\Delta_\Omega^N$ 
coincides with a predefined closed set $S\subset[0,\infty)$. If $0\in S$ the domain $\Omega$ consists of a sequence of \emph{rooms} $R_k$ connected by \emph{passages} $P_k$, each room has a \emph{wall} dividing it on two subsets connected via a \emph{door}, see Figure~\ref{fig0}. 
The diameters of $R_k$ and $P_k$ tend to zero as $k\to\infty$ in such a way that their union is a bounded domain. 

\begin{figure}[h]
\centerline{
	\begin{picture}(300,80)
	\includegraphics[width=100mm]{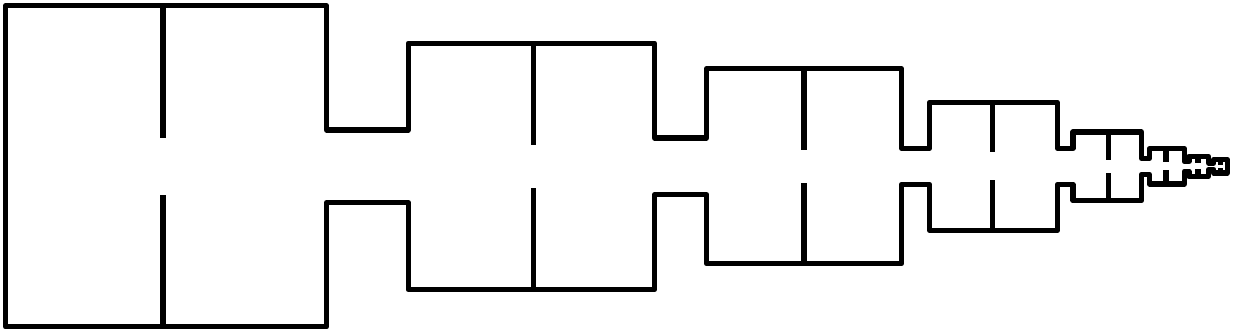}
	\put(-318,0){room}
	\put(-290,4){\vector(3,1){25}}
	\put(-314,33){door}
	\put(-290,36){\vector(1,0){43}}
	\put(-313,48){wall}
	\put(-290,51){\vector(1,0){43}}
	\put(-207,75){passage}
	\put(-200,73){\vector(0,-1){35}}
	\end{picture}}
	\caption{\label{fig0}}
\end{figure}

The strategy of the proof in \cite{HSS91} is as follows:
Choose a sequence $(s_k)_{k\in\N}$ such that 
$S=\{\text{accumulation points of }(s_k)_{k\in\N}\}$ and
consider the ``decoupled'' operator 
\begin{equation*}
\H_{\rm dec}=
\bigoplus_{k\in\N} 
\bigl(
(-\Delta_{R_k}^N) 
\oplus
(-\Delta_{P_k}^{DN})\bigr)
\end{equation*}
in the space 
$\L(\Omega)=\bigoplus_{k\in\N} \bigl(\L(R_k)\oplus\L( P_k)\bigr).$ 
Here $\Delta^N_{R_k}$ is the Neumann Laplacian on $R_k$ and 
 $\Delta^{DN}_{P_k}$ is the Laplacian on $P_k$ subject to the Dirichlet 
conditions on the parts of $\partial P_k$ touching the neighboring rooms and Neumann conditions on the remaining part of $\partial P_k$.
Denote by $(\lambda_j(-\Delta^N_{R_k}))_{k\in\N}$ and  $(\lambda_j(-\Delta^{DN}_{P_k}))_{k\in\N}$ the sequence of eigenvalues of $-\Delta_{R_k}^N$ and 
$-\Delta^{DN}_{P_k}$, respectively, numbered in ascending order with multiplicities taken into account. 
Then one has
\begin{multline}
\label{ess-sp}
\sigma_\ess(\H_{\rm dec})
=
\left\{\text{accumulation points of }(\lambda_j(-\Delta_{R_k}^N))_{j,k\in\N}\right\}\\
\cup
\left\{\text{accumulation points of }(\lambda_j(-\Delta_{P_k}^{DN}))_{j,k\in\N}\right\}.
\end{multline}
We have
\begin{gather*}
\lambda_1(-\Delta_{R_k}^N)=0,\quad
\lambda_1(-\Delta_{P_k}^{DN})=\pi^2/ (\ell(P_k))^{2},
\end{gather*}
where $\ell(P_k)$ is the length of the passage $P_k$. Next, the door in each $R_k$ is adjusted in such a way that 
\begin{gather*}
\lambda_2(-\Delta_{R_k}^N)=s_k
\end{gather*}
and it is not hard to show that 
\begin{gather}\label{lambdaR3}
\lambda_3(-\Delta_{R_k}^N)\geq C/(\mathrm{diam}(R_k))^2, 
\end{gather}
where the constant $C>0$ is the same for all rooms. 
Since $\ell(P_k)\to 0$ and $\mathrm{diam}(R_k)\to 0$  as $k\to\infty$ and $0\in S$, one concludes from \eqref{ess-sp}--\eqref{lambdaR3} that
\begin{gather*}
	\sigma_\ess(\H_{\rm dec})=\{0\}\cup\{\text{accumulation points of }(s_k)_{k\in\N}\}=\{0\}\cup S=S.
\end{gather*}
Finally, if the thickness of the passages $P_k$ tends to zero
sufficiently fast as $k\to\infty$, then the difference of the resolvents of $-\Delta_\Omega^N$ and $\H_{\rm dec}$ is a compact operator, and thus 
$\sigma_\ess(-\Delta_\Omega^N)=\sigma_\ess(\H_{\rm dec})$ by Weyl's theorem.
\smallskip

When constructing the operator $\H_\gamma$ in Theorem~\ref{th:main} we mimic the above idea. 
First of all the sequence $(z_k)_{k\in\Z}$ is split in two interlacing
subsequences $(x_k)_{k\in\Z}$ and $(y_k)_{k\in\Z}$ (see Figure~\ref{fig1}) such that 
$$\suml_{k\in\Z} d_k<\infty,\text{ where }d_k=x_{k}-x_{k-1}$$
is sufficiently small (see \eqref{d:assump1}) and  $y_k$ is the center of the interval $\I_k\coloneqq (x_{k-1},x_k)$. 
We also set
\begin{gather}
\label{Lpm}
\ell_-=-\suml_{k\in\Zm} d_{k}\quad\text{and}\quad \ell_+=\suml_{k\in\N} d_k.
\end{gather}

\begin{figure}[h]
	\centerline{
\scalebox{0.9}{	
	\begin{picture}(470,40)
	\put(0,15){\line(1,0){450}}	
	\put(0,15){\circle*{9}}
	\put(15,15){\circle*{6}}
	\put(30,15){\circle*{6}}
	\put(55,15){\circle*{6}}
	\put(100,15){\circle*{6}}
	\put(150,15){\circle*{6}}
	\put(200,15){\circle*{6}}
	\put(250,15){\circle*{6}}
	\put(300,15){\circle*{6}}
	\put(350,15){\circle*{6}}
	\put(390,15){\circle*{6}}
	\put(420,15){\circle*{6}}
	\put(440,15){\circle*{6}}
	\put(450,15){\circle*{9}}	
	\put(125,15){\circle*{2}}
	\put(175,15){\circle*{2}}
	\put(225,15){\circle*{2}}
	\put(275,15){\circle*{2}}	
	\color{black}	
	\put(-2,22){$\ell_-$}
	\put(90,22){$x_{-2}$}
	\put(140,22){$x_{-1}$}
	\put(195,22){$x_{0}$}
	\put(245,22){$x_{1}$}
	\put(295,22){$x_{2}$}
	\put(445,22){$\ell_+$}	
	\put(117,22){$y_{-1}$}
	\put(170,22){$y_{0}$}
	\put(220,22){$y_{1}$}
	\put(270,22){$y_{2}$}	
	\put(100,5){$\underset{d_{-1}}{\underbrace{\hspace{50pt}}}$}
	\put(150,5){$\underset{d_0}{\underbrace{\hspace{50pt}}}$}
	\put(200,5){$\underset{d_1}{\underbrace{\hspace{50pt}}}$}
	\put(250,5){$\underset{d_2}{\underbrace{\hspace{50pt}}}$}	
	\end{picture}}}
	\caption{\label{fig1}}
\end{figure}

For our purposes it is convenient to change the notation for the interaction strengths $\gamma_k$ as follows: at the points $y_k$ they will be denoted  by $ \al_k$,  at the points $x_k$ they will be denoted by $\be_k$, and instead of $\H_\gamma$ we will use the notation $\H_{\al,\be}$ for the Schr\"{o}dinger
operator. Now, roughly speaking, 
the intervals $\I_k $  play  the role of the rooms, the interactions at the points $x_k$  play  the role of the passages, and  the interactions at the points $y_k$   play  the role of the doors.
The desired operator is constructed in three steps.
\medskip

\noindent \textbf{1) Decoupled operator.} We start from the case $\beta_k=\infty$ for all $ k\in\Z$, which corresponds to Dirichlet decoupling at the points $x_k$. In other words, we treat the operator	
$$\H_{ \al,\infty}=\bigoplus_{k\in\Z} \Hk_{ \al_k,\I_k}\quad\text{in}\quad \L\interval=\bigoplus_{k\in\Z}\L(\I_k),$$ 
where $\Hk_{ \al_k,\I_k}$ is an operator in $\L(\I_k)$ (formally) defined by the differential expression 
\begin{gather*}
-{\d^2\over \d x^2 }+\al_k\delta(\cdot-y_k) 
\end{gather*}
and Dirichlet boundary conditions at the endpoints of $\I_k$. Recall that a sequence $S_\disc=(s_k)_{k\in \N}$ is already given and, in addition, 
we assign to $S_\ess$ a sequence $(s_k)_{k\in \Zm}$ such that
\begin{gather}\label{Sess:acc}
S_{\ess}=\left\{\text {accumulation points of }(s_k)_{k\in \Zm}\right\}.
\end{gather}

If $d_k$ are sufficiently small (see the second condition in \eqref{d:assump1}),
one can choose the constants
$ \al_k$ in such that $\lambda_1( \Hk_{ \al_k,\I_k})=s_k$ and, 
moreover,
$\lambda_2( \Hk_{ \al_k,\I_k})=\left({2\pi_k/ d_k}\right)^2$ for all $\al_k\in\R$.
Therefore, since $d_k\to 0$ as $|k|\to\infty$ and \eqref{Sprop5} holds,
we conclude
\begin{gather}
\label{acc-acc1}
\begin{split}
\sigma_\ess(\H_{\al,\infty})&
=\{\text{accumulation points of }(s_k)_{k\in\Z}\} 
\\ 
& 
=\{\text{accumulation points of }(s_k)_{k\in\Zm}\}=S_\ess.
\end{split}
\end{gather}
Similarly, if $\max_{k\in\Z}d_k$ is sufficiently small (see the first conditions in \eqref{d:assump1}), we obtain
\begin{gather*}
\sigma_\disc(\H_{\al,\infty})\cap(T_1,T_2)=
S_\disc.
\end{gather*}
Moreover, due to \eqref{Sprop4}, 
\begin{gather}\label{acc-acc3}
\text{all eigenvalues in }\sigma_\disc(\H_{\al,\infty})\cap(T_1,T_2)\text{ are simple.} 
\end{gather}
Thus, the decoupled operator $\H\aa$ satisfies \eqref{main} and \eqref{main+} in Theorem~\ref{th:main}. However, this is not the desired 
singular Schr\"{o}dinger operator as we have Dirichlet conditions at the points $x_k$, $k\in\Z$.
\medskip

\noindent 
\textbf{2) Partly coupled operator.} 
Let $\be=(\beta_k)_{k\in\N}$ be a sequence of real numbers.  
For $n\in\N $ we denote by $\H_{ \al,\be}^n$ 
the operator, which is obtained from $\H_{ \al,\infty}$ by ``inserting'' $\delta$-interactions  of strengths  $\beta_k $ at \emph{finitely many} points $x_k$, $k\in\Z\cap [-n+1,n-1]$.  
Then one has
\begin{gather}
\label{ess:n}
\forall n\in\N:\quad 
\sigma_\ess(\H\aab^n)=\sigma_\ess(\H_{ \al,\infty}),
\end{gather}
one can also  guarantee  that
the discrete spectrum of 	$\H_{ \al,\be}^n$ within $(T_1,T_2)$ changes slightly
provided $\beta_k$ are sufficiently large. 
More precisely, for an arbitrary sequence of positive numbers $(\delta_k)_{k\in\N}$
such that the neighbourhoods $[s_k-\delta_k,s_k+\delta_k]$ are pairwise disjoint and belong to $(T_1,T_2)\setminus\overline{\opset}$  one has
\begin{gather}
\label{disc:n12}
\begin{array}{l}
\sigma_{\disc}(\H\aab^n)\cap (T_1,T_2)\, \subset\, \cupl_{k\in\N}  [s_k-\delta_k,s_k+\delta_k]
\\[1mm]	
\text{and each }[s_k-\delta_k,s_k+\delta_k]\text{ contains precisely one simple eigenvalue}
\end{array}
\end{gather}
provided the entries of the sequence $\beta$ are large enough (independent of $n$).
It is important that the properties \eqref{acc-acc1} and \eqref{disc:n12} remain valid if the coefficents $\al_k$, $k\in\N$, chosen on the first step 
are slightly perturbed.

Now, we \emph{fix} a sequence $\be$ for which \eqref{disc:n12} and some additional conditions for $\beta_k$ as $|k|\to\infty$ hold; see the next step. Then one can show that for each $n\in\N$ there exist $\al_k$, $k\in\N$, that in fact
\begin{gather}
\label{disc:n3}
\begin{array}{l}
\text{for $k\in\Z\cap[1,n]$ the eigenvalue  of $\H\aab^n$ in $ [s_k-\delta_k,s_k+\delta_k]$ \textit{coincides} with $s_k$.} 
\end{array}
\end{gather}
The proof of this fact is based on a  multi-dimensional version
of the intermediate value theorem proved in \cite{HKP97}. We denote the sequence $\al$  for which \eqref{disc:n3} holds by $\al^n$.
\medskip

\noindent 
\textbf{3) Fully coupled operator.} 
Let $\al^n=(\al_k^n)_{k\in\Z}$, $n\in\N$, be the sequences from above (see the end of the previous step).
Using a standard diagonal process one concludes that there exists a sequence $\al=(\al_k)_{k\in \Z}$ such that for each $k\in\Z$
one has $\al_k^n\to\al _k$ as $n\to\infty$ (probably, up to a subsequence which is independent $k$).
As a result we get the operator $\H\aab$; it is obtained from $\H_{ \al,\infty}$ by ``inserting'' $\delta$-interactions   of the strengths  $\beta_k $ at 
\emph{all} points $x_k$, $k\in\Z$.
We prove that, if $\lim_{k\to\infty}\be_k=\infty$ and this convergence is fast enough, then also
$\sigma_\ess(\H\aab )=\sigma_\ess(\H_{ \al,\infty})$. Hence, due to \eqref{acc-acc1}, we get
\begin{gather*}
\sigma_\ess(\H\aab)=S_\ess.
\end{gather*}
Moreover, we show that
\begin{gather}
\label{nrc:n}
\H_{\al^n,\beta}^n\text{ converges to }\H_{\al ,\beta} 
\text{ in the norm resolvent sense  as }n\to\infty.
\end{gather}
Using \eqref{ess:n}--\eqref{nrc:n} we arrive at
\begin{gather*}
\sigma_\disc(\H_{\al ,\beta} )\cap(T_1,T_2)=S_\disc.
\end{gather*}
Thus the operator $\H\aab$ satisfies \eqref{main} and \eqref{main+}; we have proved Theorem~\ref{th:main}.

\subsection{Schr\"odinger operators with $\delta'$-interactions}

We mention that a result similar to Theorem~\ref{th:main} can also be proved for singular Sch\"odinger operators with 
$\delta'$-interactions of the form
\begin{gather}\label{delta'}
\H_\gamma'=
-{\d^2\over \d z^2}+\suml_{k\in K}\gamma_k\langle\cdot\,,\,\delta_{z_k}'\rangle\delta_{z_k}',
\end{gather}
where $K\subset\Z$, $\delta_{z_k}'$ is the distributional derivative of the  delta-function supported at $z_k\in\R$,   $\langle\phi,\delta_{z_k}'\rangle$ denotes its action 
on the test function $\phi$, and $\gamma_k\in \R\cup\{\infty\}$. In contrast to \eqref{delta-conditions} here the functions $u$ in the operator domain satisfy 
\begin{gather*}
u'(z_k -0) = u'(z_k  + 0),\quad u(z_k +0)-u(z_k -0)=\gamma_k u'(z_k\pm 0 ).
\end{gather*}

For the rigorous mathematical treatment of $\delta'$-interactions we refer to the standard monograph \cite{AGHH05}. Among the subsequent contributions  we  mention the  papers \cite{KM10,KM14}  dealing 
with the more subtle case $|z_k-z_{k-1}|\to 0$ as $|k|\to \infty$.  

One has the following counterpart 
of Theorem~\ref{th:main}. 

\begin{theorem}\label{th:main+}
Assume that the set $S_\ess\subset\R$, the sequence of real numbers $S_\disc=(s_k)_{k\in\N}$ and the interval $(T_1,T_2)\subset\R$ satisfy the conditions \eqref{Sprop1}--\eqref{Sprop5}. Moreover, let $0\in S$.
Then 
there exists a bounded interval $\interval\subset\mathbb{R}$, a sequence of points $(z_k)_{k\in\Z}$ with $z_k\in \interval$, and
a sequence of real numbers $(\gamma_k)_{k\in\Z}$ such that the operator $\H_\gamma'$ in $\L\interval$ defined by the formal expression \eqref{delta'} satisfies  the properties \eqref{main}--\eqref{main+}.
\end{theorem}

The proof of the above theorem is similar (except for some technical details) to the proof of Theorem~\ref{th:main}, therefore we only 
sketch it briefly. Note that in \cite[Section~3]{BK21} the	``$\sigma_\ess$-part'' of Theorem~\ref{th:main+} was already shown (if $S\subset[0,\infty)$).
Again we split the sequence   $(z_k)_{k\in\Z}$ in \eqref{delta'} in two interlacing
subsequences $(x_k)_{k\in\Z}$ and $(y_k)_{k\in\Z}$, where $y_k$ is in the center of $\I_k=(x_{k-1},x_k)$.
Instead of $\gamma_k$ we denote the interaction strengths at the points $x_k$ by $\be_k$ and at the points $y_k$ by $\al_k$.
We shall write $\H_{\al,\be}'$ instead of $\H_\gamma'$ for the corresponding Schr\"{o}dinger operator.

In the first step we set  $\beta_k=\infty$, which corresponds to Neumann decoupling at the points 
$x_k$, and we consider 
$$\H_{ \al,\infty}'=\bigoplus_{k\in\Z} \Hk'_{ \al_k,\I_k}\quad\text{in}\quad \L\interval=\bigoplus_{k\in\Z}\L(\I_k).$$ 
Here $\Hk'_{ \al_k,\I_k}$ is the operator in $\L(\I_k)$ (formally) defined by the differential expression 
\begin{gather*}
-{\d^2\over \d x^2 }+\al_k \langle\cdot\,,\,\delta_{z_k}'\rangle\delta_{z_k}',
\end{gather*}
and Neumann boundary conditions at the endpoints of $\I_k$.
In contrast to the operator $\Hk_{ \al_k,\I_k}$, the spectrum 
of $\Hk'_{ \al_k,\I_k}$ \emph{always} contains the eigenvalue $0$ (and leads to the additional condition $0\in S$ in Theorem~\ref{th:main+}). 
Moreover, if $d_k=x_k-x_{k-1}$ is sufficiently small, 
one can choose $\al_k$ in such that
\emph{the first nonzero eigenvalue} of $\Hk'_{ \al_k,\I_k}$ 
coincides with $s_k$  (recall that $s_k$ are the elements of the sequence $S_\disc$ as $k\in\N$, while for $k\in\Zm$ the numbers $s_k$ are defined in \eqref{Sess:acc}).  
\emph{The second nonzero eigenvalue} of $\Hk'_{ \al_k,\I_k}$ is larger or equal to $ \pi^2d_k^{-2}$. 
Thus the properties \eqref{acc-acc1}--\eqref{acc-acc3} hold for the decoupled operator $\H_{ \al,\infty}'$.

In the second step one perturbs $\H_{ \al,\infty}'$ by ``inserting'' $\delta'$-interactions 
of strengths  $\be_k $ at finitely many points $x_k$, $k\in\Z\cap [-n+1,n-1]$. This perturbation does not change the
essential spectrum, while the discrete spectrum will change slightly provided the constants $\be_k$ are sufficiently large.
Moreover, varying $\al_k$ one can even achieve a \emph{precise coincidence} of the discrete spectrum within
$(T_1,T_2)$ with a prescribed sequence $S_\disc$. 

In the last step we pass to the limit $n\to\infty$ and prove that the above properties remain 
valid if 
$\beta_k\to\infty$ as $|k|\to\infty$ sufficiently fast (see the condition (3.18) in \cite{BK21}).

\subsection{Structure of the paper}

The paper is organized as follows. 
In Section~\ref{sec:single} we recall the definition and some spectral properties of Schr\"odinger operators with a single 
$\delta$-interaction on a bounded interval.
The decoupled operator  $\H_{\al,\infty}$ is treated in Section~\ref{sec:deco} and the rigorous definition of the coupled operator $\H\aab$
and a precise formulation of our main result are contained in Section~\ref{sec:co}. 
In Section~\ref{sec:ess} we describe the essential spectrum of $\H\aab$. 
Section~\ref{sec:paco} is devoted to the partly coupled operator  operator $\H^n\aab$ 
and 
its spectral properties. In Section~\ref{sec:disc} we describe the discrete spectrum of $\H\aab$ and complete the proof of our main result.
Finally, in Appendix A we collect some useful material the direct sum of semibounded closed forms and associated self-adjoint operators.

\section{Single $\delta$-interaction on a bounded interval\label{sec:single}}

In this section we recall the definition of Schr\"odinger operators  
on a bounded interval with a $\delta$-interaction supported at an internal point of the interval and
either Dirichlet (this is the most important case for our constructions), Neumann or Robin boundary conditions; 
we also establish some spectral properties of these operators. For more details on $\delta$-interactions we refer to \cite[Section~I.3]{AGHH05}.

Throughout this paper $\lambda_j(\mathcal{H})$ denotes the $j$th eigenvalue of a self-adjoint 
operator $\mathcal{H}$ with purely discrete spectrum bounded from below and accumulating at $\infty$; as usual 
the eigenvalues are counted with multiplicities and ordered as a nondecreasing sequence.
In the following let $\I=(x_-,x_+)\subset\R$ be a bounded interval of length $d(\I)=x_+-x_-$ and 
middle point $y={x_-+x_+\over 2}$.
For $\u,\,\vv\in \W^{1,2}(\I)$, $\al,\,\beta_-,\,\beta_+\in\R$ we consider
\begin{gather}
\label{R-expression}
\begin{split}
\hk_{\I,\al}[\u,\vv]&=(\u',\vv')_{\L(\I)}+\al\, \u(y)\overline{\vv(y)},
\\
\hk_{\I,\al,\be_-,\be_+}[\u,\vv]&=\hk_{\I,\al}[\u,\vv]+
{1\over 2}\left(\be_- \u(x_- )\overline{\vv(x_- )} + \be_+ \u (x_+ )\overline{\vv (x_+ )}\right)
\end{split}
\end{gather} 
(recall that $\W^{1,2}(\I)\subset \C(\overline{\I})$, that is, the values of $\u,\,\vv$ at $x_-,\,x_+,\,y$ are well-defined).

\subsection{Endpoints with Dirichlet boundary conditions\label{sec:single:D}}

In the space $\L(\I)$ we introduce the densely defined, symmetric sesquilinear form
$\hk^D_{\I,\al}$ by
\begin{gather*}
\hk^D_{\I,\al}[\u,\vv]=\hk_{\I ,\al}[\u,\vv],\qquad \dom(\hk^D_{\I,\al})=\W_0^{1,2}(\I).
\end{gather*}
A standard form perturbation argument shows that this form is bounded from below and closed in $\L(\I)$, and the induced norm 
$$
\|\u\|_{\hk^D_{\I,\al}}\coloneqq\left(\hk^D_{\I,\al}[\u,\u]-C \|\u\|^2_{\L(\I)}+\|\u\|^2_{\L(\I)}\right)^{1/2}, 
\text{ where }
C =\inf_{\|\u\|_{\L(\I)}=1}\hk^D_{\I,\al}[\u,\u],
$$
on  $\dom(\hk^D_{\I,\al})$
is equivalent to the usual norm on $\W_0^{1,2}(\I)$.
Therefore, by the first representation theorem  (see, e.g. \cite[Chapter 6, Theorem 2.1]{K66}) there exists a unique self-adjoint operator $\Hk^D_{\I,\al}$ in $\L(\I)$
such that $\dom( \Hk^D_{\I,\al})\subset\dom( \hk^D_{\I,\al})$ and
\begin{gather}\label{FRT}
(\Hk^D_{\I,\al} \u,\vv)_{\L(\I)}=\hk^D_{\I,\al}[\u,\vv],\qquad \u\in \dom(\Hk^D_{\I,\al}),\, \vv\in\dom(\hk^D_{\I,\al}).
\end{gather}

By inserting  {suitable test functions $\vv$ in \eqref{FRT}}
one can easily derive the following explicit characterization of the domain and the action of $\Hk^D_{\I,\al}$.

\begin{proposition}\label{prop:HD:1}
The self-adjoint operator $\Hk^D_{\I,\al}$ associated to the form $\hk^D_{\I,\al}$ via \eqref{FRT} is given by
\begin{equation*}
\begin{split}
(\Hk^D_{\I,\al} \u)\restriction_{\I\setminus\{y\}}&= -(\u\restriction_{\I\setminus\{y\}})'',\\
 \dom(\Hk^D_{\I,\al})&=\left\{\u\in \W^{2,2}\left(\I\setminus\{y\}\right):\ \begin{array}{l}
\u(x_-)=\u(x_+)=0,\\ 
\u(y -0)=\u(y +0),\\
\u'(y+0)-\u'(y -0)=\al \u( y\pm 0 )
\end{array}
\right\}.
\end{split}
\end{equation*}
\end{proposition}

By  {Rellich's theorem}
the space $(\dom(\hk^D_{\I,\al}),\|\cdot\|_{\hk^D_{\I,\al}}) $ is compactly embedded  
in $\L(\I)$;
recall that the norm $\|\cdot\|_{\hk^D_{\I,\al}}$ is equivalent to the usual norm on $\W_0^{1,2}(\I)$.
Hence (see, e.g., \cite[Proposition~10.6]{Sch12}) the spectrum of the operator $\Hk^D_{\I,\al}$ is purely discrete. 
Using Proposition~\ref{prop:HD:1} one can easily calculate all eigenvalues of $\Hk^D_{\I,\al}$. 
In order to formulate the related statement in Proposition~\ref{prop:HD:2} below we introduce for $d>0$ the set 
$$\Pi^D_d=\left\{\left({2\pi k\over d}\right)^2 : k\in\N \right\}$$
and the function $\FF^D_d:\R\setminus\Pi^D_d \to {\R}$,
\begin{gather}
\label{F}
\FF^D_d(\lambda)=
\begin{cases}\ds{-2 \sqrt{\lambda} \cot\left({ d  \sqrt{\lambda}\over 2}\right)},&\lambda>0,
\\
-\ds{4\over d},&\lambda=0,
\\
-\ds{2 \sqrt{-\lambda}\coth\left({ d \sqrt{-\lambda}\over 2}\right)},&\lambda<0.
\end{cases}    
\end{gather}
For any fixed $d>0$ the function $\FF^D_d$ is continuous and monotonically increasing on each connected component of $\R\setminus\Pi^D_d$, moreover we have
\begin{gather}
\label{FD:prop}
\lim_{\lambda\to-\infty}\FF^D_{d}(\lambda)=-\infty\quad\text{and}\quad \lim_{\lambda\to \mu\mp 0}\FF^D_{d}(\lambda)=\pm\infty\text{ for all }\mu\in \Pi^D_d.
\end{gather}
In particular, for $d>0$ fixed and any $\al\in\R$ the equation $\al=\FF^D_d(\lambda)$ has a unique solution in $(-\infty,\,(2\pi/d )^2)$; 
this solution will be denoted by $\Lambda^D_{\al,d}$ in the following.

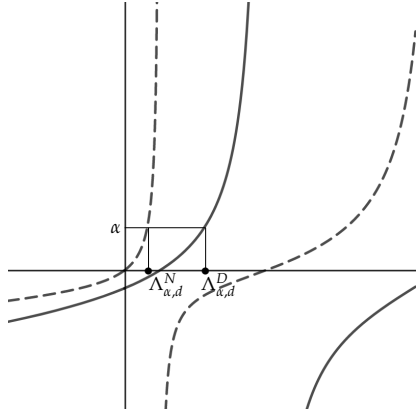
\begin{figure}[h]
\scalebox{0.9}{
\begin{picture}(100,180)
\includegraphics[width=60mm]{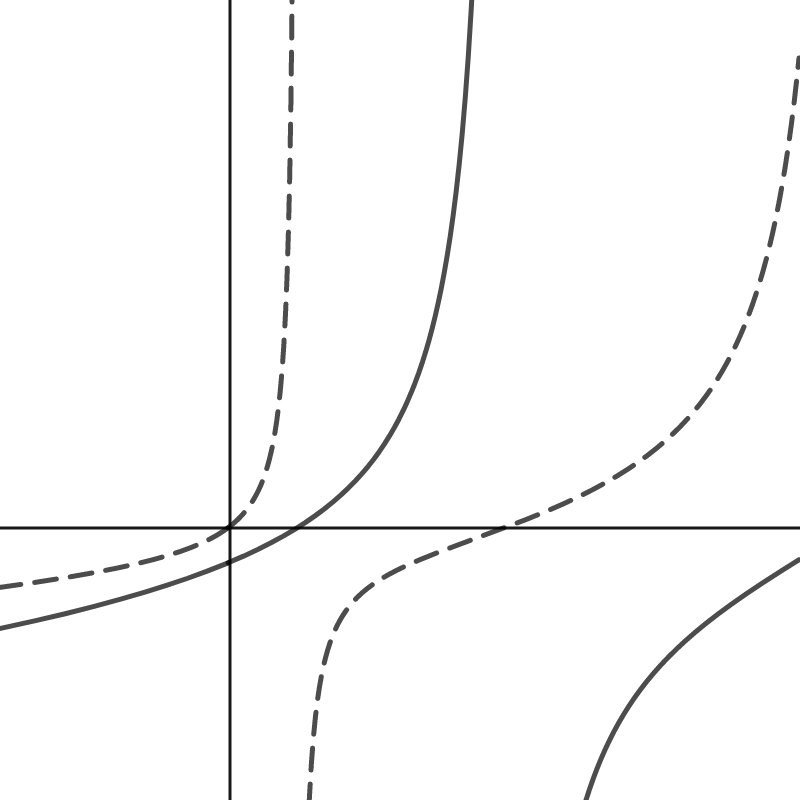}

\put(-88,76){\line(0,-1){18}}
\put(-88,76){\line(-1,0){34}}
\put(-112,76){\line(0,-1){18}}
\put(-128,70){$^{\al}$} 
\put(-88,58){\circle*{3}}
\put(-112,58){\circle*{3}}
\put(-90,45){$^{\Lambda_{\al,d}^D}$}
\put(-112,45){$^{\Lambda_{\al,d}^N}$}

\end{picture}}
\caption{The functions $\FF_{d }^D$ (solid plot) and $\FF_{d }^N$ (dashed plot), and the unique solutions $\Lambda^D_{\al,d}$ 
of $\al=\FF^D_d(\lambda)$ 
in $(-\infty,\,(2\pi/d )^2)$ and $\Lambda^N_{\al,d}$  of $\al=\FF^N_d(\lambda)$ in
$(-\infty,\,(\pi/d )^2)$.}\label{fig2}
\end{figure}

\begin{proposition}
\label{prop:HD:2}
The spectrum of the self-adjoint operator $\Hk^D_{\I,\al}$ is given by 
\begin{equation*}
 \sigma(\Hk^D_{\I,\al})=\Pi^D_{d(\I)}\cup\bigl\{\lambda\in\R\setminus\Pi^D_{d(\I)} : \FF^D_{d(\I)}(\lambda)=\al\bigr\}
 \end{equation*}
 and one has
\begin{gather*}
\lambda_{1}(\Hk^D_{\I,\al})
=\Lambda^D_{\al,d(\I)},\quad
\lambda_{2}(\Hk_{\I,\al}^D) =\left({2\pi  \over d(\I)}\right)^2.
\end{gather*}
\end{proposition}

\subsection{Endpoints with Neumann boundary conditions\label{sec:single:N}}

Besides the form $\hk^D_{\I,\al}$ we also consider
the densely defined, symmetric sesquilinear form
\begin{gather*}
\hk^N_{\I,\al}[\u,\vv]=\hk_{\I ,\al}[\u,\vv],\qquad \dom(\hk^D_{\I,\al})=\W^{1,2}(\I),
\end{gather*}
in $\L(\I)$. This form is also bounded from below and closed, and hence there exists a unique self-adjoint operator $\Hk^N_{\I,\al}$ in $\L(\I)$
such that $\dom( \Hk^N_{\I,\al})\subset\dom( \hk^N_{\I,\al})$ and
\begin{gather}\label{FRTN}
(\Hk^N_{\I,\al} \u,\vv)_{\L(\I)}=\hk^N_{\I,\al}[\u,\vv],\qquad \u\in \dom(\Hk^N_{\I,\al}),\, \vv\in\dom(\hk^N_{\I,\al}).
\end{gather}
The counterpart of Proposition~\ref{prop:HD:1} in the present situation reads as follows.

\begin{proposition}\label{prop:HN:1}
The self-adjoint operator $\Hk^D_{\I,\al}$ associated to the form $\hk^N_{\I,\al}$ via \eqref{FRTN} is given by
\begin{equation*}
\begin{split}
(\Hk^N_{\I,\al} \u)\restriction_{\I\setminus\{y\}}&= -(\u\restriction_{\I\setminus\{y\}})'',\\
 \dom(\Hk^N_{\I,\al})&=\left\{\u\in \W^{2,2}\left(\I\setminus\{y\}\right):\ \begin{array}{l}
\u'(x_-)=\u'(x_+)=0,\\ 
\u(y -0)=\u(y +0),\\
\u'(y+0)-\u'(y -0)=\al \u( y\pm 0 )
\end{array}
\right\}.
\end{split}
\end{equation*}
\end{proposition}

It follows that 
the spectrum of $\Hk^N_{\I,\al}$ is purely discrete. In a similar way as in the previous subsection 
we consider for $d>0$ the set
$$\Pi^N_d=\left\{\left({\pi (2k-1)\over d}\right)^2 :k\in\N \right\}$$
and the function
\begin{gather*}
\FF^N_d(\lambda)=
\begin{cases}\ds{2 \sqrt{\lambda} \tan\left({ d  \sqrt{\lambda}\over 2}\right)},&\lambda>0,
\\
0,&\lambda=0,
\\
-\ds{2 \sqrt{-\lambda}\tanh\left({ d \sqrt{-\lambda}\over 2}\right)},&\lambda<0.
\end{cases}    
\end{gather*}
The function $\FF^N_d$ is continuous and monotonically increasing on each connected component of $\R\setminus\Pi^N_d$, and 
the properties in
\eqref{FD:prop} hold also $\FF^N_d$. 
In particular, for $d>0$ fixed and any $\al\in\R$ the equation
$\al=\FF^N_d(\lambda)$ has a unique solution $\Lambda^N_{\al,d}$ in $(-\infty,\,(\pi/d )^2)$; cf. Figure~\ref{fig2}.

\begin{proposition}\label{prop:HN:2}
The spectrum of the self-adjoint operator $\Hk^N_{\I,\al}$ is given by 
\begin{equation*}
 \sigma(\Hk^N_{\I,\al})=\Pi^N_{d(\I)}\cup\bigl\{\lambda\in\R\setminus\Pi^N_{d(\I)} : \FF^N_{d(\I)}(\lambda)=\al\bigr\}
 \end{equation*}
 and one has
\begin{gather*}
\lambda_{1}(\Hk^N_{\I,\al})
=\Lambda^N_{\al,d(\I)},\quad
\lambda_{2}(\Hk_{\I,\al}^N) =\left({\pi  \over d(\I)}\right)^2.
\end{gather*}
\end{proposition}

\subsection{Endpoints with Robin boundary conditions}\label{sec:single:R}

In this subsection we consider the densely defined, symmetric sesquilinear form
\begin{gather*}
\hk^R_{\I,\al,\be_-,\be_+}[\u,\vv]=\hk_{\I,\al,\be_-,\be_+}[\u,\vv],\qquad
\dom(\hk^R_{\I,\al,\be_-,\be_+})=\W^{1,2}(\I),
\end{gather*}
which is again bounded from below and closed in $\L(\I)$; we shall use later that 
the induced norm on  $\dom(\hk^R_{\I,\al,\be_-,\be_+})$
is equivalent to the usual norm on $\W^{1,2}(\I)$.
The corresponding self-adjoint operator $\Hk^R_{\I,\al,\be_-,\be_+}$ in $\L(\I)$
has the following form.

\begin{proposition}\label{prop:HR:1}
The self-adjoint operator $\Hk^R_{\I,\al,\be_-,\be_+}$ associated to the form $\hk^R_{\I,\al,\be_-,\be_+}$ is given by

\begin{equation*}
\begin{split}
(\Hk^R_{\I,\al,\be_-,\be_+} \u)\restriction_{\I\setminus\{y\}}&= -(\u\restriction_{\I\setminus\{y\}})'',\\
 \dom(\Hk^R_{\I,\al,\be_-,\be_+})&=\left\{\u\in \W^{2,2}\left(\I\setminus\{y\}\right):\ \begin{array}{l}
\u'(x_-)={1\over 2}\beta_-\u(x_{-}),\\ 
\u'(x_+)=-{1\over 2}\beta_+\u(x_{+}),\\ 
\u(y -0)=\u(y +0),\\
\u'(y+0)-\u'(y -0)=\al \u( y\pm 0 )
\end{array}
\right\}.
\end{split}
\end{equation*}
\end{proposition}

The spectrum of the operator $\Hk^R_{\I,\al,\be_-,\be_+}$ is purely discrete. 
It can be determined in a similar way as in the previous subsections. However, the precise values 
are not needed for our purposes. Instead, we make use of the fact that
 for large $\be_\pm$ the eigenvalues of $\Hk^R_{\I,\al,\be_-,\be_+}$ are close to the eigenvalues of the self-adjoint operator
$\Hk^D_{\I , \al}$.

\begin{proposition}\label{prop:HR:2}
For the $j$-th eigenvalue of the self-adjoint operators $\Hk^R_{\I,\al,\be_-,\be_+}$ and $\Hk^D_{\I , \al}$ one has
\begin{gather}\label{prop:HR:2:conv}
\lambda_j(\Hk^R_{\I , \al ,\be_-,\be_+ })\to\lambda_j(\Hk^D_{\I , \al})\text{
as }\min\{\beta_-,\beta_+\}\to+\infty.
\end{gather}
\end{proposition}

\begin{proof}
First we note that by the min-max principle (see, e.g., \cite[Section~4.5]{De95}) the function
$$\R^2\ni(\beta_-,\beta_+)\mapsto \lambda_j(\Hk^R_{\I , \al ,\be_-,\be_+ })$$
is monotonically increasing in each of its arguments.
Therefore it suffices to show that
\begin{gather}\label{prop:HR:2:conv2}
\lambda_j(\Hk^R_{\I , \al ,\be,\be})\to\lambda_j(\Hk^D_{\I , \al})\text{
as }\beta\to+\infty.
\end{gather}
Without loss of generality we may assume in the following that $\beta\ge 0$.
Note first that for $\beta\leq \widetilde \beta$ we have $\hk^R_{\I , \al ,\be,\be}\leq \hk^R_{\I , \al ,\widetilde\be,\widetilde\be}$ in the (usual) form sense.
In the present situation this means 
\begin{equation*}
 \hk^R_{\I , \al ,\be,\be}[u,u]\leq \hk^R_{\I , \al ,\widetilde\be,\widetilde\be}[u,u],\qquad 
 u\in \dom(\hk^R_{\I , \al ,\widetilde\be,\widetilde\be})=\dom(\hk^R_{\I , \al ,\be,\be}).
\end{equation*}
Moreover, it is easy to see that 
\begin{gather*}
\dom(\hk^D_{\I , \al})=
\left\{
u\in\bigcap\limits_{\beta\ge 0} \dom(\hk^R_{\I , \al ,\be ,\be }):\ 
\sup_{\beta\ge 0} 
\hk^R_{\I , \al ,\be ,\be  }[u,u]<\infty
\right\}
\end{gather*}
and that $\hk^R_{\I , \al ,\be ,\be  }[u,u]=\hk^D_{\I , \al}[u,u]$ holds for all $u\in \dom(\hk^D_{\I , \al})$.
Therefore, \cite[Theorem~3.1]{S78} (see also \cite[Theorem 4.2]{BHSW10})  implies the strong resolvent convergence
\begin{gather}\label{src}
\forall f\in \L(\I):\quad
\left\|(\Hk^R_{\I , \al ,\be ,\be  }-\mu\Id)^{-1}f-
(\Hk^D_{\I , \al }-\mu\Id)^{-1}f\right\|_{\L(\I)}\to 0\text{ as }\be\to+\infty,
\end{gather}
where $\mu\in \rho(\Hk^R_{\I , \al,\be,\be})\cap \rho(\Hk^D_{\I , \al})$, and as usual $\Id$ stands for the identity operator.
Now observe that we can chose $\mu<\min\{\sigma(\Hk^D_{\I , \al}),\sigma(\Hk^R_{\I , \al,\be,\be})\}$, $\beta\geq 0$. 
Then $\hk^R_{\I , \al ,\be ,\be  }\leq \hk^D_{\I , \al}$ implies
\begin{equation*}
 (\Hk^D_{\I , \al  }-\mu\Id)^{-1}\leq (\Hk^R_{\I , \al ,\be ,\be  }-\mu\Id)^{-1},
\end{equation*}
and since the resolvents
$(\Hk^R_{\I , \al ,\be ,\be  }-\mu\Id)^{-1}$ and $(\Hk^D_{\I , \al  }-\mu\Id)^{-1}$ are both compact in $\L(\I)$,
we conclude from \cite[Theorem VIII-3.5]{K66}  that the strong convergence in 
\eqref{src} becomes even convergence in the operator norm, i.e. 
\begin{gather}\label{nrc} 
\left\|(\Hk^R_{\I , \al ,\be ,\be  }-\mu\Id)^{-1}-
(\Hk^D_{\I , \al }-\mu\Id)^{-1}\right\|\to 0\text{ as }\beta\to +\infty.
\end{gather}
It is well-known (see, e.g., \cite[Corollary~A.15]{P06}) that  
\eqref{nrc} implies  \eqref{prop:HR:2:conv2}, and hence \eqref{prop:HR:2:conv}.
\end{proof}

\section{The decoupled operator $\H\aa$\label{sec:deco}}

In this section we define and study the spectrum of the self-adjoint operator
$$\H\aa=\bigoplus_{k\in\Z} \Hk^D_{\I_k,\al_k} \quad\text{in}\quad %\L\interval=
\bigoplus_{k\in\Z}\L(\I_k)$$
with suitably chosen interaction strengths $\al_k$ and intervals $\I_k$ that are stacked in a row with finite total length.

\subsection{Auxiliary sequence $\widehat S_\ess$ and the intervals $\I_k$}

Recall that the set $S_{\ess} $ and the sequence  $S_{\disc}=(s_k )_{k\in\N}$ satisfying \eqref{Sprop1}--\eqref{Sprop5} are given. 
It is easy to see that, due to \eqref{Sprop1}--\eqref{Sprop2},
one can always find a sequence $\widehat S_\ess=(s_k)_{k\in\Zm}$
such that
	\begin{gather}\label{Sprop1+} 
	S_{\ess}=\{\text{accumulation points of }\widehat S_\ess\},
	\\
	\label{Sprop2+} 
	 \widehat S_\ess\cap [T_1,T_2]\subset \opset.
	\end{gather}
Note that the elements of $S_\disc$ and $\widehat S_\ess$ are both denoted by $s_k$, 
but $s_k$ with index $k\in\N$ belongs to $S_\disc$, while $s_k$ with index $k\in\Zm$ is an element of $\widehat S_\ess$.
Recall from \eqref{Sprop3} that the sequence  $S_{\disc}=(s_k )_{k\in\N}$ is contained in $(T_1,T_2)$. For all $k\in\Z$ we 
fix $d_k>0$ such that
\begin{gather}\label{d:assump1}
T_2<\min_{k\in \Z}(\pi/d_k)^2\qquad\text{and}\qquad s_k<(\pi/d_k)^2,\quad k\in\Z,
\end{gather}
and we assume, in addition, that the numbers $d_k$ satisfy
\begin{gather}\label{d:assump2}
\sum_{k\in\Z} d_k<\infty.
\end{gather}
In particular, this implies
\begin{gather}\label{dto0}
d_k\to 0\text{ as }k\to \pm\infty.
\end{gather}
Note that for $k\in\N$ the first condition in \eqref{d:assump1} implies the second condition since $s_k\in (T_1,T_2)$ for $k\in\N$, and hence the second condition 
is only needed for $d_k$ (and $s_k$) with index $k\in\Zm$.
Finally, we set (see Figure~\ref{fig1})
$$\I_k=(x_{k-1},x_k),\quad k\in\Z,$$
where 
\begin{gather*}
x_{0}=0,\quad x_k=
\left\{\begin{array}{lll}
  x_{k-1}+d_{k},&k\in\N,\\[1mm]
  x_{k+1}-d_{k+1},&k\in\mathbb{Z}\setminus(\N\cup \{0\}).
\end{array}\right.
\end{gather*}
The intervals $\I_k$ satisfy
$\cup_{k\in\Z}\,\overline{\I_k}=[\ell_-,\ell_+]$ with
$\ell_\pm$ in \eqref{Lpm}.
Due to \eqref{d:assump2} the interval 
$[\ell_-,\ell_+]$ is compact.

\subsection{Choice of the interaction strengths $\al_k$}

In  what  follows  we denote by $B_\delta(s)$ the open $\delta$-neighbourhood of $s\in\R$, i.e.
$$B_\delta(s)=(s-\delta,s+\delta).$$
Let us fix a sequence $ (\delta_k)_{k\in\N}$  of positive numbers with the properties
\begin{align}
\label{delta1}
&  
\overline{B_{\delta_k}(s_k)}\subset (T_1,T_2)\setminus\overline{\opset},\quad k\in\N,\\
&  \label{delta2}
\overline{B_{\delta_k}(s_k)}\cap \overline{B_{\delta_k}(s_l)}=\varnothing,\quad k\not=l.
\end{align}
The above choice of $\delta_k$ is always possible due to 
\eqref{Sprop3} and \eqref{Sprop:add1}.
Moreover, we claim that $\delta_k$ can be chosen so small that
\begin{gather}
\label{a-a}
\al_k^+-\al_k^-\leq c_k d_k\text{ with }c_k\to 0\text{ as }k\to\infty,
\end{gather}
where
\begin{gather}
\label{alpm}
\al_k^\pm=\FF_{d_k}^D(s_k\pm{1\over 2}\delta_k).
\end{gather}
In fact, \eqref{a-a} holds for $\delta_k$ small enough 
since the function $\FF_{d_k}^D$ in \eqref{F} is  continuous on $ (-\infty,(2\pi/ d_k)^2)$, and 
\begin{gather}
\label{sk:est}
s_k+{1\over 2}\delta_k<s_k+\delta_k<T_2<\min_{k\in \Z}(\pi/d_k)^2
\end{gather}
(cf.~\eqref{d:assump1}, \eqref{delta1}). 
Condition \eqref{a-a} will be required only in the last step of our construction; cf. Lemma~\ref{lemma:nrc}.
Note, that by \eqref{sk:est}
\begin{gather}\label{al<0}
\al_k^\pm<0,\qquad k\in\N,
\end{gather}
since $\FF_{d_k}^D((\pi/ d_k)^2)=0$ and 
$\FF_{d_k}^D$ is strictly increasing on the interval $(-\infty,(2\pi/ d_k)^2)$.

From now we shall work with sequences $(\al_k)_{k\in\Z}$ that satisfy the following hypothesis.

\begin{hypothesis}\label{hypo31}
The notation $\al=(\al_k)_{k\in\Z}$ is used for a sequence of real numbers with the properties 
\begin{gather}\label{al:cond}
\begin{array}{ll}
\al_k\in [\al_k^-,\al_k^+],& k\in\N,\\[2mm]
\al_k=\FF^D_{d_k}(s_k),&k\in\Z\setminus\N.
\end{array}
\end{gather}
\end{hypothesis}

\subsection{The decoupled operator $\H\aa$}\label{decosec}
Let $\al=(\al_k)_{k\in\Z}$ be any sequence satisfying Hypothesis~\ref{hypo31}.
For each $k\in\Z$ 
we consider the sesquilinear form $\hk^D_{\I_k,\al_k}$ and the associated self-adjoint operator $\Hk^D_{\I_k,\al_k}$ as in Section~\ref{sec:single:D}
with $\al$, $\I$, $x_-$, $x_+$, and $y$ replaced by
$\al_k$, $\I_k$, $x_{k-1}$, $x_k$
and $y_k={x_{k-1}+x_k\over 2}$, respectively.
The spectrum of $\Hk^D_{\I_k,\al_k}$ is discrete and by
Proposition~\ref{prop:HD:2} the first eigenvalue
$\lambda_1(\Hk^D_{\I_k,\al_k})$ is the unique solution of the equation
$\al_k=\FF^D_{d_k}(\lambda)$ in $(0,(2\pi/ d_k)^2)$. Therefore, taking into account the monotonicity and continuity of the function 
$\FF^D_{d_k}$, \eqref{alpm}, and \eqref{sk:est}  we conclude that
\begin{align}
\label{first-lambda1}
&\lambda_1(\Hk^D_{\I_k,\al_k})\in \overline{B_{\delta_k/2}(s_k)},\quad\, k\in\N,\\
\label{first-lambda2}
&\lambda_1(\Hk^D_{\I_k,\al_k})=s_k,\qquad\qquad k\in\Zm,
\end{align}
and Proposition~\ref{prop:HD:2} also gives
\begin{gather} 
\label{second-lambda}
\lambda_2(\Hk^D_{\I_k,\al_k})=\left({2\pi\over d_k}\right)^2,\qquad k\in\Z.
\end{gather}
It is clear from \eqref{first-lambda1} and \eqref{first-lambda2} that  
the forms $\hk^D_{\I_k,\al_k}$ are bounded from below by 
\begin{gather}
\label{sinf}
s_{\inf}=\inf\left\{\inf_{k\in\N} (s_k-\delta_k/2);\, \inf_{k\in\Zm} s_k\right\},
\end{gather}
and from \eqref{Sprop1} and \eqref{delta1} we conclude $s_{\inf}\geq \min\{T_1;\,\min S_\ess\}>-\infty.$

Following Appendix~\ref{A3} we consider the densely defined, semibounded, closed form
$$
\h\aa=\bigoplus_{k\in\Z}\,\hk^D_{\I_k,\al_k}\quad\text{in}\quad \L\interval=\bigoplus_{k\in\Z} \L(\I_k)
$$
and the corresponding self-adjoint operator $\H\aa$. In the present situation Proposition~\ref{aprop3+} and Theorem~\ref{th:A:spec} lead to the next statement.

\begin{theorem}\label{bigthm1}
Let $\al=(\al_k)_{k\in\Z}$ be a sequence satisfying Hypothesis~\ref{hypo31}. Then the self-adjoint operator $\H\aa$ associated to the form $\h\aa$
in $\L\interval$
is given by
\begin{equation*}
\begin{split}
(\H\aa u)\restriction_{\I_k}&=\Hk^D_{\I_k,\al_k} \u_k=-(\u_k\restriction_{\I_k\setminus\{y_k\}})'',\\
\dom(\H\aa)&=\left\{u\in\L\interval:\ \u_k\in\dom(\Hk^D_{\I_k,\al_k})\text{ and }
\ds\suml_{k\in\Z}\|\Hk^D_{\I_k,\al_k}\u_k\|^2_{\L(\I_k)}<\infty \right\},
\end{split}
\end{equation*}
where $\u_k$ stands for the restriction of $u$ on $\I_k$. Furthermore,  one has
\begin{gather}\label{th:deco1} 
\sigma_\ess(\H\aa)=S_\ess.
\end{gather}
\end{theorem}

\begin{proof}
It is clear from Proposition~\ref{aprop3+} and Proposition~\ref{prop:HD:1} that the self-adjoint operator $\H\aa$ is given as in the theorem. Hence it remains
to verify \eqref{th:deco1}. In fact, by Theorem~\ref{th:A:spec} we have
\begin{gather}
\label{oplus-spec1}
\sigma_{\ess}(\H\aa)=\{\text{accumulation points of }S\},
\end{gather}
where $S=(\lambda_j(\Hk^D_{\I_k,\al_k}))_{j\in\N,k\in\Z}$ is the sequence of all eigenvalues of the operators
$\Hk^D_{\I_k,\al_k}$. 
By \eqref{dto0} and \eqref{second-lambda} one has
\begin{gather*}
\lambda_j(\Hk^D_{\I_k,\al_k})\to\infty\text{ as }|k|\to\infty,\qquad j\geq 2,
\end{gather*}
and using \eqref{Sprop1+} and \eqref{first-lambda2} we get
\begin{gather*}\label{prop-ess2}
\left\{\text{accumulation points of }(\lambda_1(\Hk^D_{\I_k,\al_k}))_{k\in\Z\setminus\N}\right\}=S_\ess.
\end{gather*}
Finally, it follows easily from \eqref{delta1}, \eqref{delta2}, \eqref{first-lambda1} that
$$\left\{\text{accumulation points of }(\lambda_1(\Hk^D_{\I_k,\al_k}))_{k\in\N}\right\}=
\bigl\{\text{accumulation points of }(s_k)_{k\in\N}\bigr\}.$$
Therefore, taking into account \eqref{Sprop5}, we get
\begin{gather}\label{prop-ess3}
\left\{\text{accumulation points of }(\lambda_1(\Hk^D_{\I_k,\al_k}))_{k\in\N}\right\}\subset S_\ess.
\end{gather}
Combining \eqref{oplus-spec1}--\eqref{prop-ess3} 
we arrive at \eqref{th:deco1}.
\end{proof}

\begin{remark}
If $\al=(\al_k)_{k\in\Z}$ is a sequence satisfying Hypothesis~\ref{hypo31} and $\H\aa$ is the self-adjoint operator in the previous theorem then one also has
\begin{equation*}
\sigma_{\disc}(\H\aa)\cap (T_1,T_2)=
\left\{\lambda_1(\Hk^D_{\I_k,\al_k}):\ k\in\N\right\}
\end{equation*}
and each of these eigenvalues is simple, i.e. 
$\dim\ker(\H\aa-\lambda_1(\Hk^D_{\I_k,\al_k})\Id)=1$; these facts follow again from 
Proposition~\ref{prop:HD:2}, \eqref{delta1}, \eqref{delta2}, \eqref{first-lambda1}, \eqref{second-lambda} and Theorem~\ref{th:A:spec}. 
In particular, if $\al_k=\FF_{d_k}^D(s_k)$ for \textbf{all} $k\in\Z$, then
\begin{gather*}
\sigma_{\disc}(\H\aa)\cap (T_1,T_2)=S_\disc.
\end{gather*}
\end{remark}

The following lemma on the semiboundedness of the forms $\hk_{\I_k, \al_k}$ will be used later.

\begin{lemma}\label{neumannchen}
Let $\al=( \al_k)_{k\in\Z}$ be a sequence satisfying Hypothesis~\ref{hypo31}. Then there exists a constant
$C>0$ which depends only on the quantity $s_{\inf}$ in \eqref{sinf} and the interval $\interval$
such that for all $k\in\Z$ 
\begin{gather}\label{nonnegative}
\hk_{\I_k, \al_k}[\u,\u]+{C\over d_k^2}\|\u\|^2_{\L(\I_k)}\geq 0,\qquad  \u\in\W^{1,2}(\I_k).
\end{gather}  
\end{lemma}
 
 \begin{proof}
Consider the function $\FF_{d_k}^D$ defined by \eqref{F} and recall that it is monotonically increasing on 
$(-\infty,(2\pi/ d_k)^2)$ and $\FF_{d_k}^D((\pi/ d_k)^2)=0$. 
It follows from Hypothesis~\ref{hypo31}, the choice of $d_k$ in \eqref{d:assump1}, and \eqref{sinf} that  
\begin{gather}
\label{alfa-est1}
0>\alpha_k\geq \FF_{d_k}^D( s_{\inf}).
\end{gather}
In the case $s_{\inf}\ge 0$ we have the estimate 
\begin{gather*}
\FF_{d_k}^D(s_{\inf})\ge \FF_{d_k}^D(0)=-{4\over d_k},
\end{gather*}
and in the case $s_{\inf}< 0$ we make use of the fact that
the function $x\mapsto x\coth(x)$ is increasing on $[0,\infty)$ and $d_k< \ell_+-\ell_-$ to derive
\begin{equation}\label{alfa-est3}
\begin{split}
\FF_{d_k}^D(s_{\inf})&=
-{4\over d_k}\cdot {d_k\sqrt{-s_{\inf}}\over 2}\coth {d_k\sqrt{-s_{\inf}}\over 2}
\\
&\ge
-{4\over d_k}\cdot {(\ell_+-\ell_-)\sqrt{-s_{\inf}}\over 2}\coth {(\ell_+-\ell_-)\sqrt{-s_{\inf}}\over 2}.
\end{split}
\end{equation}
Combining \eqref{alfa-est1}--\eqref{alfa-est3} we conclude that
there exists a constant $\widehat C>0$ which depends only on $s_{\inf}$ and $\interval$ such that
\begin{gather}
\label{alpha-inf}
\alpha_k\geq -{\widehat C\over d_k},\qquad k\in\Z.
\end{gather}
By Proposition~\ref{prop:HN:2} we have
\begin{gather*}
\lambda_1(\Hk^N_{\I_k, \al_k})=\Lambda^N_{\al_k,d_k},
\end{gather*}
where $\Lambda^N_{\al_k,d_k}$ is the unique solution of $\al_k=\FF^N_{d_k}(\lambda)$ in $(-\infty,(\pi/d_k)^2)$. 
Using \eqref{alpha-inf} and the monotonicity of the function $\FF^N_{d}$ we obtain 
\begin{gather*}
\lambda_1(\Hk^N_{\I_k, \al_k})\geq\Lambda^N_{-{\widehat C\over d_k},d_k}=-{4\widehat\lambda^2\over d_k^2},
\qquad\text{where}\,\,\,\widehat\lambda=\frac{d_k}{2}\sqrt{-\Lambda^N_{-{\widehat C\over d_k},d_k}}>0
\end{gather*} 
is the unique solution of 
$4\widehat\lambda\tanh\widehat\lambda=\widehat C$.
From this we finally conclude 
that there exists $C>0$   
such that \eqref{nonnegative} holds.
 \end{proof}

\section{The coupled operator $\H\aab$\label{sec:co}}

In this short section we introduce the self-adjoint operator $\H\aab$ in $\L\interval$ and reformulate our main result Theorem~\ref{th:main} in a more rigorous form.
The operator $\H\aab$ corresponding to the form $\h\aab$ below will be our main object of interest in this paper, it can be viewed as
perturbation of the decoupled operator $\H\aa$ in the sense that  {all} neighbouring intervals $\I_{k}$ and $\I_{k+1}$ are glued together via 
$\delta$-couplings of sufficiently large interaction strengths.

Let $\al=(\al_k)_{k\in\Z}$ be a sequence which satisfies Hypothesis~\ref{hypo31} and 
let $\be=(\be_k)_{k\in\Z}$ be another sequence of real numbers.
In the space $\L\interval$ we introduce the symmetric sesquilinear form $\h\aab$ by  
\begin{equation}\label{999}
\begin{split}
 \h\aab[u,v]&=\suml_{k\in\Z} \hk_{\I_k, \al_k,\be_{k-1},\be_k}[\u_k,\vv_k],\\
\dom(\h\aab)&=\left\{
u\in \W^{1,2}_{\rm loc}\interval\cap\L\interval :\
\suml_{k\in\Z}|\hk_{\I_k, \al_k,\be_{k-1},\be_k}[\u_k,\u_k]|<\infty
\right\},
\end{split}
\end{equation}
where $\u_k=u\restriction_{\I_k}$, $\vv_k= v\restriction_{\I_k}$, and the forms $\hk_{\I_k, \al,\be_{k-1},\be_k}$  are defined as in \eqref{R-expression}.

\begin{lemma}\label{lemma:clo}
There exists a  sequence $\be^{\inf}=(\be_k^{\inf})_{k\in\Z}$ of real numbers such that 
for any sequence $\be=(\beta_k)_{k\in\Z}$ satisfying
\begin{gather}\label{beta:large}
\beta_k^{\inf}\leq\beta_k<\infty
\end{gather}
and any sequence $ \al=( \al_k)_{k\in\Z}$ satisfying Hypothesis~\ref{hypo31}
the form $\h\aab$ is densely defined, closed, and bounded from below by
$s_{\inf}-1$, where $s_{\inf}$ is the quantity specified in \eqref{sinf}. 
\end{lemma}

\begin{proof}
For $k\in\Z$ consider the densely defined, closed, semibounded form $\hk^R_{\I_k, \al_k,\be_{k-1},\be_k}$ and the associated self-adjoint operator
$\Hk^R_{\I_k, \al_k,\be_{k-1},\be_k}$ as in Section~\ref{sec:single:R},
with $\al$, $\I$, $\beta_-$, and $\beta_+$ replaced by
$\al_k$, $\I_k$, $\be_{k-1}$, and $\be_k$, respectively.
By Proposition~\ref{prop:HR:2}   we have
\begin{gather}
\label{EV-conv12}
\lambda_1(\Hk^R_{\I_k,  \al_k,\be_{k-1},\be_k})\nearrow\lambda_1(\Hk^D_{\I_k,  \al_k }) \text{ as }\min\{\beta_{k-1},\beta_{k}\}\to\infty
\end{gather}
for any $k\in\Z$ and hence we conclude from
\eqref{EV-conv12} that there exists a sequence $\be^{\inf}=(\be_k^{\inf})_{k\in\Z}$ such that 
$$\lambda_1(\Hk^R_{\I_k,  \al_k,\be_{k-1},\be_k})\geq \lambda_1(\Hk^D_{\I_k,  \al_k })-1.$$
Taking into account \eqref{first-lambda1}--\eqref{first-lambda2}, we get
\begin{gather*}%\label{lowerbound}
\hk^R_{\I_k, \al_k,\be_{k-1},\be_k}[\u,\u] \geq (s_{\inf}-1)\|\u\|^2_{\L(\I_k)},\qquad \u\in \W^{1,2}(\I_k),\,k\in  \Z,
\end{gather*}
for any sequence $\be=(\beta_k)_{k\in\Z}$ satisfying \eqref{beta:large} and any sequence $ \al=( \al_k)_{k\in\Z}$ satisfying Hypothesis~\ref{hypo31}.

Following Appendix~\ref{A3} we introduce in $\L\interval$ the densely defined, closed form 
$$\widetilde\h\aab= \bigoplus_{k\in\Z } \hk^R_{\I_k, \al_k,\be_{k-1},\be_k},$$
that is,
\begin{equation*}
\begin{split}
 \widetilde\h\aab[u,v]&=\suml_{k\in\Z} \hk^R_{\I_k, \al_k,\be_{k-1},\be_k}[\u_k,\vv_k],\\
\dom(\widetilde\h\aab)&=\left\{
u\in \L\interval:\,\u_k\in \W^{1,2}(\I_k),
\suml_{k\in\Z}|\hk^R_{\I_k, \al_k,\be_{k-1},\be_k}[\u_k,\u_k]|<\infty
\right\},
\end{split}
\end{equation*}
which is bounded from below by $s_{\inf}-1$. Observe that the form $\h\aab$ in \eqref{999} is the restriction of $\widetilde\h\aab$ onto 
 \begin{gather}\label{forms-doms}
\dom(\h\aab)=\dom(\widetilde\h\aab)\cap\W^{1,2}_{\rm loc}\interval.
\end{gather}
This implies that $\h\aab$ is also bounded from below by $s_{\inf}-1$ and {it is clear from \eqref{999}} that $\h\aab$ is densely defined in $\L\interval$.
It remains to verify that $\h\aab$ is closed. For this consider $\dom(\h\aab)$ equipped with the form norm 
$$
\|v\|_{ \al,\be} =\left(\h\aab[v,v]-(s_{\inf}-1)\|v\|^2_{\L\interval}+\|v\|^2_{\L\interval}\right)^{1/2}
$$
and let $(u^n)_{n\in\N}$ be a Cauchy sequence in $\dom(\h\aab)$ with respect to this norm. Then $(u^n)_{n\in\N}$ 
is also a Cauchy sequence in $\dom(\widetilde\h\aab)$ equipped with the form norm
$$\|v\|_{\widetilde{ \al,\be}} =\left(\widetilde\h\aab[v,v]-(s_{\inf}-1)\|v\|^2_{\L\interval}+\|v\|^2_{\L\interval}\right)^{1/2}$$
and as $\widetilde\h\aab$ is closed there exists a limit $u\in\dom(\widetilde\h\aab)$. 
It is clear that for each $k\in\Z$ the restrictions $(\u_k^n)_{n\in\N}$ are Cauchy sequences in $\dom(\hk^R_{\I_k, \al_k,\be_{k-1},\be_k})$
equipped with the corresponding form norm, which is equivalent to the usual $\W^{1,2}(\I_k)$-norm, and we have 
$\u_k^n\rightarrow \u_k$ 
as $n\rightarrow\infty$ in $\W^{1,2}(\I_k)$, where $\u_k=u\restriction_{\I_k}$.
Then  by the trace theorem  
\begin{gather}\label{conv:n}
\u^n_k(x_k-0)\to \u_k(x_k-0)\quad\text{and}\quad \u^n_k(x_{k-1}+0)\to \u_k(x_{k-1}+0)\text{ as }n\to\infty.
\end{gather}
Since $u^n\in \W^{1,2}_{\rm loc}\interval$ we have the continuity condition
\begin{gather*}
\u^n_k(x_{k-1}+0)=\u^n_{k-1}(x_{k-1}-0) ,\qquad k\in\Z,
\end{gather*}
and  {together with \eqref{conv:n} and $\u_k\in \W^{1,2}(\I_k)$ for each $k\in\Z$}, we conclude
\begin{gather}\label{trace-eq}
u\in \W^{1,2}_{\rm loc}\interval.
\end{gather}
It follows from \eqref{forms-doms} and  \eqref{trace-eq}  that $u\in\dom (\h\aab)$, thus 
the form $\h\aab$ is closed.
\end{proof}

Besides Hypothesis~\ref{hypo31} we will also assume that next hypothesis is satisfied.

\begin{hypothesis}\label{hypo41}
The sequence 
$\be=(\beta_k)_{k\in\Z}$ satisfies the condition $\beta_k^{\inf}\leq\beta_k<\infty$ in \eqref{beta:large} and moreover,
without loss of generality, we assume that 
\begin{gather}
\label{be:geq:0}
\be_k> 0,\quad k\in\Z.
\end{gather}
It is clear that any sequence $\widehat\be=(\widehat\be_k)_{k\in\Z}$ such that $\be_k\le \widehat\be_k$ for all $k\in\Z$ also satisfies 
\eqref{beta:large} and \eqref{be:geq:0}.
\end{hypothesis}

The self-adjoint operator associated to the form $\h\aab$ is denoted by $\H\aab$. 
It is not difficult to verify the next statement.

\begin{proposition}
Let $\al=( \al_k)_{k\in\Z}$ and $\be=( \be_k)_{k\in\Z}$ be sequences satisfying Hypothesis~\ref{hypo31} and Hypothesis~\ref{hypo41}.
Then the self-adjoint operator $\H\aab$ in $\L\interval$  associated to the form $\h\aab$ is 
bounded from below by $s_{\inf}-1$ and is 
given by
\begin{equation*}
\begin{split}
(\H\aab u)\restriction_{\I_k\setminus\{y_k\}}&= 
(- \u_k\restriction_{\I_k\setminus\{y_k\}})'',\\
 \dom(\H\aab)&=\left\{ \begin{matrix} u=(\u_k)_{k\in\Z}\in \L\interval, \\ \u_k\in \W^{2,2}\left(\I_k\setminus\{y_k\}\right)\end{matrix}:\ \begin{matrix}
u(y_k +0)=u(y_k -0),\\
u(x_k +0)=u(x_k -0),\\
\, u'(y_k+0)-u'(y_k -0)= \al_k u( y_k\pm 0 ),\\
u'(x_k+0)-u'(x_k -0)=\be_k u( x_k\pm 0 )\,
\end{matrix}
\right\}.
\end{split}
\end{equation*}
\end{proposition}

Now we are ready to formulate a rigorous version of our main result Theorem~\ref{th:main}.

\begin{theorem}[Main Theorem] 
\label{th:main:pre}
There exist  sequences  $\al=(  \al_k)_{k\in\Z}$ and $\be=(\be_k)_{k\in\Z}$ satisfying 
Hypothesis~\ref{hypo31} and Hypothesis~\ref{hypo41} such that the essential spectrum of the self-adjoint operator $\H\aab$
coincides with $S_{\ess}$, 
\begin{gather}\label{main-ess}
 \sigma_{\ess}(\H\aab)=S_{\ess},
\end{gather}
the discrete spectrum in $(T_1,T_2)$ coincides with $S_{\disc}$,
\begin{gather}\label{main-disc}
 \sigma_{\disc}(\H\aab)\cap (T_1,T_2)= S_{\disc},
\end{gather}
and each eigenvalue $\lambda\in \sigma_{\disc}(\H\aab)\cap (T_1,T_2)$ is simple, i.e. $\dim(\ker(\H\aab-\lambda\Id))=1$.
\end{theorem}
In the next section we will prove \eqref{main-ess}. The assertion \eqref{main-disc} and the multiplicity statement will be shown in Section~\ref{sec:disc}.

\section{Essential spectrum of $\H\aab$\label{sec:ess}}

In this section we prove the identity \eqref{main-ess} in Theorem~\ref{th:main:pre}. More precisely, 
we will show that for any sequence $\al$ satisfying Hypothesis~\ref{hypo31} the essential spectra of $\H\aab$ and $\H\aa$ coincide
provided $(\be_k)^{-1}$ decays sufficiently fast as $|k|\to\infty$. Then \eqref{th:deco1} in Theorem~\ref{bigthm1} implies \eqref{main-ess}.

\begin{theorem}\label{th:ess}
Let $\al=( \al_k)_{k\in\Z}$ and $\be=( \be_k)_{k\in\Z}$ be sequences satisfying Hypothesis~\ref{hypo31} and Hypothesis~\ref{hypo41}, and assume that for 
\begin{gather*}%\label{rho}
\rho_k={1\over D_k^2}\max\left\{{1\over \be_k D_k^{3}},\ {1\over\be_{k-1} D_{k-1}^{3}}\right\},
\,\,\,\text{ where }D_k=\min\{d_k,d_{k+1}\},\,k\in\Z,
\end{gather*}
one has 
\begin{gather}\label{rho-cond}
\rho_k\to 0\text{ as }|k|\to \infty.
\end{gather}
Then the essential spectrum of the self-adjoint operator $\H\aab$ is given by 
$$\sigma_{\ess}(\H\aab)=\sigma_{\ess}(\H\aa)=S_\ess.$$ 
\end{theorem}

\begin{remark}
Condition \eqref{rho-cond} can be easily achieved by taking $\be_k$ converging fast enough to $\infty$ as $|k|\to\infty$. 
For example, \eqref{rho-cond} holds if we choose $\be_k$ such that
$$\be_k\ge {1\over \min\{D_k^{5+\epsilon},\, D_{k+1}^2D_k^{3+\epsilon}\}}\ \text{ with  } \epsilon>0.$$
\end{remark}

\begin{proof}[Proof of  Theorem \ref{th:ess}]
We set
\begin{gather}\label{mu}
\mu= s_{\inf}-2,
\end{gather}
where $s_{\inf}$ is the quantity specified in \eqref{sinf}. 
Recall that the form $\h\aa$ is bounded from below by $s_{\inf}$ and 
the form $\h\aab$ is bounded from below by $s_{\inf}-1$.
Hence $\mu$ belongs to the resolvent set of both operators $\H\aab$ and $\H\aa$, and we have 
\begin{gather}\label{mu:dist}
1\leq  \dist(\mu,\sigma(\H\aab))\quad\text{and}\quad 1 \leq \dist(\mu,\sigma(\H\aa)).
\end{gather}
We introduce the resolvent difference
$$T\aab=(\H\aab -\mu \Id)^{-1}-(\H\aa-\mu \Id)^{-1}.$$
Our goal is to prove that  
$T\aab$ is a compact operator;
then by virtue of Weyl's theorem  (see, e.g., \cite[Theorem XIII.14]{RS78})
and Theorem~\ref{bigthm1} we immediately obtain the statement of Theorem~\ref{th:ess}.

For $f,g\in \L\interval$ we set 
$$u=(\H\aab-\mu \Id)^{-1}f\quad\text{and}\quad v=(\H\aa-\mu \Id)^{-1}g.$$
Then one has 
\begin{gather*}%\label{Trepres}
\begin{split}
(T\aab f,g)_{\L\interval}
&= \bigl((\H\aab-\mu \Id)^{-1}f,g\bigr)_{\L\interval}-\bigl(f,(\H\aa-\mu \Id)^{-1}g\bigr)_{\L\interval}\\
&= \bigl(u,(\H\aa-\mu \Id)v\bigr)_{\L\interval}-\bigl((\H\aab-\mu \Id)u,v\bigr)_{\L\interval}\\
&= ( u,\H\aa v)_{\L\interval}-(\H\aab u,v)_{\L\interval}\\
&=\suml_{k\in\Z}\left(-\intl_{x_{k-1}}^{y_k}u \overline{ v''}\,\d x -\intl_{y_{k}}^{x_k}u \overline{ v''}\,\d x +
\intl_{x_{k-1}}^{y_k}u'' \overline{v}\,\d x +\intl_{y_{k}}^{x_k}u'' \overline{v}\,\d x \right)
\end{split}
\end{gather*}
 and integration by parts leads to 
\begin{gather*}
\begin{split}
(T\aab f,g)_{\L\interval}=
&-\suml_{k\in\Z}\left( u(y_{k}-0)\overline{v'(y_{k}-0)}-u(x_{k-1}+0)\overline{v'(x_{k-1}+0)}\right)
\\
&-\suml_{k\in\Z}\left( u(x_{k}-0)\overline{v'(x_{k}-0)}-u(y_{k}+0)\overline{v'(y_{k}+0)}\right)
\\
&+\suml_{k\in\Z}\left( u'(y_{k}-0)\overline{v(y_{k}-0)}-u'(x_{k-1}+0)\overline{v(x_{k-1}+0)}\right)
\\
&+\suml_{k\in\Z}\left( u'(x_{k}-0)\overline{v(x_{k}-0)}-u'(y_{k}+0)\overline{v(y_{k}+0)}\right).
\end{split}
\end{gather*}
Since both $u$ and $v$ 
satisfy the conditions 
\begin{gather*}\label{uvcondis}
\begin{split}
 u(y_k+0)=u(y_k-0),\quad & u'(y_k+0)-u'(y_k-0)=\alpha_k u(y_k\pm 0),\quad k\in\Z,\\
 v(y_k+0)=v(y_k-0),\quad & v'(y_k+0)-v'(y_k-0)=\alpha_k v(y_k\pm 0),\quad k\in\Z,
\end{split}
\end{gather*}
$v(x_k)=0$ and $u(x_{k}-0)=u(x_k+0)$ for all $k\in\Z$, the expression for $(T\aab f,g)_{\L\interval}$ reduces to
\begin{gather}\label{res-difference0}
(T\aab f,g)_{\L\interval}=\suml_{k\in\Z}u(x_k)\overline{v'(x_k+0)-v'(x_k-0)}.
\end{gather}

We introduce the operators 
$\Gamma\aab:\L\interval\to  {l^2 }(\Z)$, $\Gamma\aa:\L\interval\to  {l^2 }(\Z)$    by
\begin{gather*} 
\begin{array}{ll}
(\Gamma\aab  f)_k = D_k^{-3/2}\big[\big((\H\aab - \mu \Id)^{-1}f\big)(x_k)\big],& k\in\Z,\\[2mm]
(\Gamma\aa g)_k = D_k^{3/2}\big[ \big((\H\aa - \mu \Id)^{-1}g\big)'(x_k+0)- \big((\H\aa  - \mu \Id)^{-1}g\big)'(x_k-0)\big],& k\in\Z,
\end{array}
\end{gather*}
on their natural domains 
\begin{gather*}
\dom(\Gamma\aab )=\left\{f\in\L\interval:\ \Gamma\aab  f\in l^2(\Z)\right\},\\ \dom(\Gamma\aa)=\left\{g\in\L\interval:\ \Gamma\aa g\in l^2(\Z)\right\}.
\end{gather*}
Below we prove in Lemma~\ref{lemma:gamma} and Lemma~\ref{lemma:gamma0} that both operators $\Gamma\aab$ and $\Gamma\aa$ are bounded and everywhere defined, 
and moreover  $\Gamma\aab$ is   compact. Hence 
\eqref{res-difference0} can be rewritten as  
\begin{gather*}
(T\aab f,g)_{\L\interval}=(\Gamma\aab f,\Gamma\aa g)_{l^2(\Z)}=(\Gamma^*\aa\Gamma\aab f,g)_{\L\interval},\quad f,g\in\L\interval,
\end{gather*}
and 
we conclude that the resolvent difference
$T\aab=\Gamma^*\aa\Gamma\aab$
is compact. 

Thus to finish the proof of Theorem~\ref{th:ess} we have to prove the lemmata below.

\begin{lemma}\label{lemma:gamma}
The operator 
$\Gamma\aab$ is bounded, defined on the whole space $\L\interval$, and compact.
\end{lemma}

\begin{proof}
First we prove that $\Gamma\aab f$ is well-defined for any  $f\in\L\interval$.
As before we consider $u=(\H\aab - \mu \Id)^{-1}f\in\dom(\H\aab)\subset\dom(\h\aab)$.
Then one has for each $n\in\N$
\begin{gather}\label{Gammaf1}
 \begin{split}
 &\suml_{ k:\,|k|\leq n}|(\Gamma\aab f)_k|^2=\suml_{ k:\,|k|\leq n}{|u(x_k)|^2 \over D_k^{3}}
\\
&=\suml_{ k=-n}^{n+1} { D_k^{-3}|u(x_k)|^2+D_{k-1}^{-3}|u(x_{k-1})|^2\over 2} -\,
\frac{D_{n+1}^{-3}|u(x_{n+1})|^2}{2}\, -\, \frac{D_{-n-1}^{-3}|u(x_{-n-1})|^2}{2}\\
&\leq
\suml_{ k=-n}^{n+1}\rho_k D_k^2\left({ \be_k |u(x_k)|^2+\be_{k-1} |u(x_{k-1})|^2\over 2}\right)\\
&\leq
\suml_{ k=-n}^{n+1}\rho_k D_k^2\left( \hk_{\I_k, \al_k}[\u_k,\u_k]+{C\over d_k^2}\|\u_k\|^2_{\L(\I_k)}+{ \be_k |u(x_k)|^2+\be_{k-1} | u(x_{k-1})|^2\over 2}\right),
 \end{split}
\end{gather}
where we have used Lemma~\ref{neumannchen} and the corresponding constant $C>0$ from there in the last estimate.
Taking into account the definition of the forms $\hk_{\I_k, \al_k,\be_{k-1},\be_k}$ and $\h\aab$ we continue the above estimates by
\begin{gather}\label{long-estimate}
\begin{split}
&=
\suml_{ k=-n}^{n+1}\rho_kD_k^2\left( \hk_{\I_k, \al_k,\be_{k-1},\be_k}[\u_k,\u_k]
-\mu\|\u_k\|^2_{\L(\I_k)} + \left({C\over d_k^2}+\mu\right)\|\u_k\|^2_{\L(\I_k)}\right)\\
& \leq \max_{k\in\Z}\left(\rho_kD_k^2\right)\left(\h\aab[u,u]-\mu\|u\|^2_{\L\interval}\right)+
\max_{k\in\Z}\left(C\rho_k{D_k^2\over d_k^2}+|\mu|\rho_k D^2_k\right)\|u\|^2_{\L\interval}
\end{split}
\end{gather}
and using   $\max_{k\in\Z}\rho_k<\infty$ (this follows from \eqref{rho-cond}) and 
\begin{gather*}
D_k\le d_k<\ell_+-\ell_- 
\end{gather*}
we conclude
\begin{gather}\label{long-estimate3}
\suml_{ k:\,|k|\leq n}|(\Gamma\aab f)_k|^2\leq C_1 \left(\h\aab[u,u]-\mu\|u\|^2_{\L\interval}\right)+C_2\|u\|^2_{\L\interval},
\end{gather}
where $$
C_1=\max_{k\in\Z}\left(\rho_kD_k^2\right)\quad\text{and}\quad C_2=\max_{k\in\Z}\left(C\rho_k{D_k^2\over d_k^2}+|\mu|\rho_k D^2_k\right).
$$
Since $(\H\aab - \mu \Id)u=f$ we have 
\begin{gather}\label{einmal}
\h\aab[u,u]-\mu\|u\|^2_{\L\interval}=(f,u)_{\L\interval}
\end{gather}
and \eqref{mu:dist} implies
\begin{gather}\label{zweimal}
\|u\|_{\L\interval}=\|(\H\aab - \mu \Id)^{-1}f\|_{\L\interval}\leq \frac{1}{\dist(\mu,\sigma(\H\aab))}\|f\|_{\L\interval}\leq\|f\|_{\L\interval},
\end{gather}
so that 
%\begin{gather}\label{u<f}
%\|u\|_{\L\interval}\leq\|f\|_{\L\interval},\quad\h\aab[u,u]-\mu\|u\|^2_{\L(a,b)}=(f,u)_{\L\interval}, 
%\end{gather}
\eqref{long-estimate3} leads to
\begin{gather*}
\suml_{ k:\,|k|\leq n}|(\Gamma\aab f)_k|^2\leq (C_1+C_2)\|f\|^2_{\L\interval},
\end{gather*}
 {and hence $\|\Gamma\aab f \|_{l^2(\Z)}\leq \sqrt{C_1+C_2}\|f\|_{\L\interval}$.}
Therefore the operator $\Gamma\aab$ is bounded and well-defined on the whole space $\L\interval$.

In order to prove the compactness of $\Gamma\aab$ we consider for $n\in\N$ the finite rank operators
$$\Gamma^n\aab:\L\interval\to  {l^2}(\Z),\quad (\Gamma^n\aab f)_k=
\begin{cases}
(\Gamma\aab f)_k,& |k|\leq n,\\
0,&|k|\geq n+1.
\end{cases}
$$
Then for $f\in\L\interval$ one has 
\begin{gather}\label{gaga1}
\|\Gamma^n\aab f - \Gamma\aab  f\|^2_{l^2(\Z) }=
\suml_{|k|\geq n+1}  D_k^{-3}|u(x_k)|^2,
\end{gather}
where as before $u=(\H\aab-\mu \Id)^{-1}f$.
Repeating verbatim the arguments in  \eqref{Gammaf1} and \eqref{long-estimate} we obtain
\begin{equation}\label{gaga1+}
\begin{split}
\suml_{|k|\geq n+1}  D_k^{-3}|u(x_k)|^2&\leq 
\max_{|k|\ge n+1}\left(\rho_k D_k^2\right) \left(\h\aab[u,u]-\mu\|u\|^2_{\L\interval}\right)\\
&\qquad\qquad +
\max_{|k|\ge n+1}\left(C\rho_k{D_k^2\over d_k^2}+|\mu|\rho_k D_k^2\right) \|u\|^2_{\L\interval}.
\end{split}
\end{equation}
From  \eqref{rho-cond} it is clear that $\max_{|k|\ge n+1}\rho_k\to 0$ as $n\to\infty$ and hence
\begin{gather}\label{nach000}
 \max_{|k|\ge n+1}\left(\rho_k D_k^2\right)\to 0\quad\text{and}\quad \max_{|k|\ge n+1}\left(C\rho_k{D_k^2\over d_k^2}+|\mu|\rho_k D_k^2\right)\to 0
\end{gather}
as $n\to\infty$. Using  \eqref{einmal} and \eqref{zweimal} we conclude for \eqref{gaga1} from \eqref{gaga1+}--\eqref{nach000}
\begin{gather}\label{gaga1++}
 \|\Gamma^n\aab f - \Gamma\aab  f\|^2_{l^2(\Z) }=
\suml_{|k|\geq n+1}  D_k^{-3}|u(x_k)|^2\leq C_n\|f\|^2_{\L\interval}
\end{gather}
with $C_n\rightarrow 0$ as $n\to\infty$.
Thus $\Gamma^n\aab\to\Gamma\aab$ in the operator norm as $n\to \infty$. Since $\Gamma^n\aab$ are finite rank operators, 
the operator $\Gamma\aab$ is  compact.  
\end{proof}

\begin{lemma}\label{lemma:gamma0}
The operator 
$\Gamma_{\al,\infty}$ is bounded and defined on the whole space $\L\interval$.
\end{lemma}

\begin{proof}
For $g\in\L\interval$ we consider $v=(\H\aa-\mu\Id)^{-1}g\in\dom(\H\aa)\subset\dom(\h\aa)$. 
For $k\in\Z$ we denote by $\w_k$ the linear function defined on $\I_k$ which has the value $0$ at left endpoint $x_{k-1}$ and the value $1$ at right endpoint $x_k$, 
i.e.
$$\w_k (x)={x-x_{k-1}\over x_k-x_{k-1}},\qquad x\in\I_k.$$
Integrating by parts, taking into account that $v$  {satisfies $v(y_k+0)=v(y_k-0)$,} and using the notation $\vv_k=v\restriction_{\I_k}$ we get
\begin{gather*}
\begin{split}
\bigl((\H\aa v)\upharpoonright_{\I_k},\w_k \bigr)_{\L(\I_k)}&=-\intl_{x_{k-1}}^{y_k}\vv_k'' \w_k \,\d x -\intl_{y_{k}}^{x_k}\vv_k'' \w_k \,\d x \\
&=
\intl_{x_{k-1}}^{x_k}\vv'_k \w'_k\, \d x + \al_k \vv_k(y_k)\w_k(y_k) - \vv'_k(x_k-0),
\end{split}
\end{gather*}
which leads to 
\begin{gather} \label{v-repr}
\begin{split}
\vv_k'(x_k-0)&=\intl_{x_{k-1}}^{x_k}\vv'_k \w'_k\, \d x + \al_k \vv_k(y_k)\w_k(y_k) - \bigl((\H\aa v)\upharpoonright_{\I_k},\w_k \bigr)_{\L(\I_k)}\\
&=\left[ \hk_{\I_k,\al_k}[\vv_k,\w_k]+
{C\over d_k^2}(\vv_k,\w_k)_{\L(\I_k)}\right]-\bigl((\H\aa v)\upharpoonright_{\I_k},\w_k \bigr)_{\L(\I_k)}- 
{C\over d_k^2}(\vv_k,\w_k)_{\L(\I_k)},
\end{split}
\end{gather}
where
$C$ is the positive constant in Lemma~\ref{neumannchen}. It is easy to compute 
\begin{gather*}%\label{wk-prop}
\|\w_k\|^2_{\L(\I_k)}= {d_k\over 3}\quad\text{and}\quad
 \hk_{\I_k,\al_k}[\w_k,\w_k]={1\over d_k}+{\alpha_k\over 4}\leq { 1\over d_k},
\end{gather*} 
where \eqref{al<0} was used in the last estimate.
Due to Lemma~\ref{neumannchen} we then obtain the following Cauchy-Schwarz inequality and corresponding estimate
\begin{gather}\label{CS}
 \begin{split}
  &\left| \hk_{\I_k,\al_k}[\vv_k,\w_k]+{C\over d_k^2}(\vv_k,\w_k)_{\L(\I_k)}\right|^2\\
  &\qquad \leq
\left( \hk_{\I_k,\al_k}[\vv_k,\vv_k]+{C\over d_k^2}\|\vv_k\|^2_{\L(\I_k)}\right)
\left( \hk_{\I_k,\al_k}[\w_k,\w_k]+{C\over d_k^2}\|\w_k\|^2_{\L(\I_k)}\right)\\
&\qquad \leq 
\frac{C'}{d_k}
\left( \hk_{\I_k,\al_k}[\vv_k,\vv_k]+{C\over d_k^2}\|\vv_k\|^2_{\L(\I_k)}\right)
 \end{split}
\end{gather}
with $C'>0$.
Combining \eqref{v-repr} and \eqref{CS} and taking into account that $d_k<\ell_+-\ell_-$ we arrive at the estimate  
\begin{gather}\label{vk-0}
|\vv_k'(x_k-0)|^2\leq 
C''\left( d_k^{-1}  \hk_{\I_k,\al_k}[\vv_k,\vv_k]+ 
 \|(\H\aa v)\upharpoonright_{\I_k}\|^2_{\L(\I_k)}+d_k^{-3}\|\vv_k\|^2_{\L(\I_k)}\right).
\end{gather}
With $\w_k$ replaced by $1-\w_k$ we obtain with the same arguments the estimate
\begin{gather*}
|\vv_k'(x_{k-1}+0)|^2\leq 
C'''\left( d_k^{-1}  \hk_{\I_k,\al_k}[\vv_k,\vv_k]+ 
 \|(\H\aa v)\upharpoonright_{\I_k}\|^2_{\L(\I_k)}+d_k^{-3}\|\vv_k\|^2_{\L(\I_k)}\right).
\end{gather*}
Now we obtain
\begin{gather*}%\label{gamma0est}
 \begin{split}
 \|\Gamma\aa g\|^2_{l^2(\Z)}
 & = \suml_{k\in\Z}  D_k^3 \vert v'(x_k-0) -  v'(x_k+0)\vert^2\\
 &\leq 2 \suml_{k\in\Z}  D_k^3|\vv_k'(x_k-0)|^2 + 2 \suml_{k\in\Z}  D_{k-1}^3|\vv_k'(x_{k-1}+0)|^2
 \end{split}
\end{gather*}
and using $D_{k-1}\le d_k<\ell_+-\ell_-$ and $D_k\le d_k<\ell_+-\ell_-$, $k\in\Z$, we conclude 
\begin{gather*}
\|\Gamma\aa g\|^2_{l^2(\Z)} \leq
\widetilde C\left(\h\aa[v,v]+\|\H\aa v\|^2_{\L\interval}+\|v\|^2_{\L\interval}\right)
\end{gather*}
with some $\widetilde C>0$.
Since $(\H\aa-\mu\Id)v=g$,
$$
\h\aa[v,v]=(\H\aa v,v)_{\L\interval}=(g+\mu v,v)_{\L\interval},$$
and 
$\|v\|_{\L\interval}\leq\|g\|_{\L\interval}$ by \eqref{mu:dist} (see also \eqref{zweimal}), 
we   obtain finally
\begin{gather}\label{gamma0est:final}
 \|\Gamma\aa g\|^2_{l^2(\Z)}
  = \suml_{k\in\Z}  D_k^3 \vert v'(x_k+0) -  v'(x_k-0)\vert^2
 \leq  C \|g\|^2_{\L\interval}.
 \end{gather}
This shows that $\Gamma\aa$ is bounded and  well-defined on the whole space $\L\interval$.
\end{proof}
\end{proof}

\section{Partly coupled operators  and their spectra \label{sec:paco}}

In this section we take another step towards the proof of the main result in this paper. Here our objective is to study a partly coupled operator 
$\H\aab^n$ for $n\in\N$ that is obtained from the decoupled operator $\H_{ \al,\infty}$ by introducing finitely many $\delta$-couplings of strengths $\beta_k$ 
at the points $x_k$, $k=-n+1,\dots,n-1$. After some technical preparations in Section~\ref{sec61} it will be shown in Theorem~\ref{th:n:essdisc} that 
one can pick sequences $\alpha^n$ and $\beta$ such that $\sigma_{\ess}(\H_{\al^n,\be}^n)=S_\ess$
and $\sigma_{\disc}(\H_{\al^n,\be}^n)\cap (T_1,T_2)=
S_\disc$, that is, Theorem~\ref{th:main:pre} holds for the partly coupled operator $\H\aab^n$.

In the following we use the notation
\begin{gather*}
\mathcal{K}=\bigl\{k\in \Zm:\ s_k\in [T_1,T_2]\bigr\}
\end{gather*}
and for $n\in\N$ we set 
\begin{gather*}
\mathcal{K}^n=\bigl\{k\in \{-n+1,\dots,0\}:\ s_k\in [T_1,T_2]\bigr\}.
\end{gather*}
Recall that we have already fixed a sequence $(\delta_k)_{k\in\N}$ which satisfies \eqref{delta1}--\eqref{a-a}.
In addition, we now introduce a sequence $(\delta_k)_{k\in\Zm}$ such that
\begin{gather}\label{delta3}
B_{\delta_k}(s_k)\subset\begin{cases} \opset, & k\in\mathcal K, \\ \R\setminus[T_1,T_2], & k\in\Z\setminus(\N\cup\mathcal{K}),\end{cases}
%B_{\delta_k}(s_k)\subset \opset,\quad k\in\mathcal K,\qquad\text{and}\qquad  B_{\delta_k}(s_k)\subset \R\setminus[T_1,T_2],\quad k\in\Z\setminus(\N\cup\mathcal{K}),
\end{gather}
which is possible since $s_k\in\opset$ for $k\in\mathcal K$ (see \eqref{Sprop2+}) and $s_k\in \R\setminus[T_1,T_2]$ for $k\in\Z\setminus(\N\cup\mathcal{K})$. To avoid technical difficulties below 
we sometimes discuss the situation $\mathcal{K}=\Z\setminus\N$, in which case $\mathcal{K}^n=\{-n+1,\dots,0\}$; in other words we treat the situation $s_k\in (T_1,T_2)$
for all $k\in\Z$, which appears if $S_{\ess}\subset [T_1,T_2]$.

\subsection{The operator $\Hk\aab^n$ and its spectrum}\label{sec61}

Let again $\al=(\alpha_k)_{k\in \Z}$  be a sequence satisfying Hypothesis~\ref{hypo31}, 
while $\be=(\be_k)_{k\in\Z}$ is a sequence of real numbers.  
For $n\in\N $ we consider the densely defined, closed, symmetric sesquilinear form
\begin{gather*}
 \begin{split}
  \hk^n_{\alpha,\beta}[u,v]&=\sum_{k=-n+1}^{n}\hk_{\I_k,\alpha_k}[\u_k,\vv_k]+\sum_{k=-n+1}^{n-1}\beta_k u(x_k)\overline{v(x_k)},\\
  \dom (\hk^n_{\alpha,\beta})&=\W^{1,2}_0(x_{-n},x_{n}),
 \end{split}
\end{gather*}
in $\L(x_{-n},x_{n})$ and the corresponding self-adjoint operator $\Hk\aab^n$ in $\L(x_{-n},x_{n})$ with $\delta$-interactions 
of strengths $\beta_k$ and $\alpha_k$ on 
$$\mathcal Z_k\coloneqq \bigl\{x_k:-n+1\leq k\leq n-1 \bigr\} \cup \bigl\{y_k:-n+1\leq k\leq n \bigr\},$$
which is given by
\begin{equation*}
\begin{split}
(\Hk\aab^n u)\restriction_{\I_k\setminus\{y_k\}}&= 
(- \u_k\restriction_{\I_k\setminus\{y_k\}})'',\qquad -n+1\le k\le  n,\\
 \dom(\Hk\aab^n)&=\left\{ 
 \begin{matrix}
  u\in \W^{2,2}((x_{-n},x_{n})\setminus \mathcal Z_k)\\
  u(x_{-n})=u(x_{n})=0
  %\u_k\in \W^{2,2}(\I_k\setminus\{y_k\})
 \end{matrix} :
 \begin{matrix} 
% u\in \W^{2,2}\left((x_{-n},x_{n})\setminus(\bigcup_{k=-n+1}^{n-1}\{x_k\}\bigcup_{k=-n+1}^{n}\{y_k\})\right)\\
u(y_k +0)=u(y_k -0),\\%\, -n+1\leq k \leq n,\\
u(x_k +0)=u(x_k -0),\\ %-n+1\leq k \leq n-1\\
\,u'(y_k+0)-u'(y_k -0)= \al_k u( y_k\pm 0 ),\\
u'(x_k+0)-u'(x_k -0)=\be_k u( x_k\pm 0 ),\\
\text{for all}\,\,x_k,y_k\in\mathcal Z_k 
\end{matrix}
\right\}.
\end{split}
\end{equation*}
It is clear that 
$\Hk\aab^n$ is independent of $\al_k$ with $k\not\in \{-n+1,\dots,n\} $ and $\be_k$ with $k\not\in\{-n+1,\dots,n-1\}$. Furthermore, 
the spectrum of $\Hk\aab^n$ is purely discrete.

 The main result of this subsection is Theorem~\ref{th:nspec} for which some preparatory
statements are needed.
 The first lemma shows that 
the eigenvalues of $\Hk\aab^n$ are close to the eigenvalues of
the self-adjoint operator 
\begin{gather}\label{finitedecoup}
\bigoplus_{ k=-n+1}^{n}\Hk^D_{\I_k,\al_k},
\end{gather}
provided $\beta_k$ are sufficiently large. Moreover, the eigenvalues satisfy the additional useful inequalities \eqref{induction:pm}, where we consider 
sequences  $\al=((\underline\al^{n,k})_{l})_{l\in\Z}$, $\al=((\overline\al^{n,k})_{l})_{l\in\Z}$, $k=1,\dots,n$, that satisfy Hypothesis~\ref{hypo31} and are of the particular form \eqref{special:sequence}
with the numbers $\al_k^\pm$ defined by \eqref{alpm}.
The inequalities \eqref{induction:pm} will be needed later, when applying the intermediate value theorem from \cite{HKP97} in the proof of Lemma~\ref{lemma:exact:n}.

\begin{lemma}\label{lemma:n:disc:0} 
There exists a sequence $\be'=(\be'_k)_{k\in\Z}$ of real numbers such that the following holds: 
for any sequence $\al=(\alpha_k)_{k\in \Z}$ 
satisfying Hypothesis~\ref{hypo31},  
for any sequence $\be=(\be_k)_{k\in\Z}$ satisfying  $\be_k'\le\beta_k$, $k\in\Z$, and for any
$n\in\N$,
one has
\begin{gather} 
\label{induction}
\sigma(\Hk\aab^n)\cap (T_1,T_2)=
\left\{s^n_{\al,\be;\, k} : k\in\mathcal{K}^n\cup \{1,\dots,n\}\right\},
\end{gather}
where $s^n_{\al,\be;\, k}\in B_{\delta_k}(s_k)$ are simple eigenvalues of $\Hk\aab^n$. Moreover,
for sequences $\underline\al^{n,k}$ and $\overline\al^{n,k}$, $k=1,\dots,n$, 
that satisfy Hypothesis~\ref{hypo31} and are of the particular form
 \begin{gather}\label{special:sequence}
 \begin{split}
  \underline\al^{n,k}&=\bigl(\dots,\alpha_{-1},\alpha_0,\alpha_1^+,\dots,\alpha_{k-1}^+,\alpha_k^-,\alpha_{k+1}^+,\dots,\alpha_n^+,\alpha_{n+1},\dots\bigr),\\
  \overline\al^{n,k}&=\bigl(\dots,\alpha_{-1},\alpha_0,\alpha_1^-,\dots,\alpha_{k-1}^-,\alpha_k^+,\alpha_{k+1}^-,\dots,\alpha_n^-,\alpha_{n+1},\dots\bigr)
  \end{split}
  \end{gather}
one has 
\begin{gather}
\label{induction:pm}
s^n_{\underline\al^{n,k},\be;\, k}< s_k-{1\over 4}\delta_k,\quad
s_k+{1\over 4}\delta_k< s^n_{\overline\al^{n,k},\be;\, k},\quad k=1,\dots,n.
\end{gather}
\end{lemma}

\begin{proof}
To avoid further technical difficulties we assume that $\mathcal{K}=\Zm$, which leads to $\mathcal{K}^n=\{-n+1,\dots,0\}$. Hence 
\begin{gather}\label{wlog}
s_k\in\opset\subset(T_1,T_2),\qquad k\in\Zm,
\end{gather}
and,
in particular, it follows from \eqref{wlog} that $S_\ess\subset [T_1,T_2]$. The general case needs only slight modifications.

We prove \eqref{induction} and \eqref{induction:pm} below by induction. 
For convenience, we restrict ourselves to the sequences $\beta=(\beta_k)_{k\in\Z}$ such that
\begin{gather}\label{be:restriction}
\beta_k\geq \be^{\inf}_k\quad\text{and}\quad \beta_k>0 
\end{gather}
(in another words, $\beta$ satisfies Hypothesis~\ref{hypo41}).
This assumption implies, in particular,
 that for each $n\in\N$ the form $\hk^n_{\alpha,\beta}$ is bounded from below by $s_{\inf}-1$. In fact, for 
$u\in \dom (\hk^n_{\alpha,\beta})=\W^{1,2}_0(x_{-n},x_{n})$ and its extension $\widetilde u$ by zero on all of $\interval$ one has 
$\widetilde u\in \dom (\hk_{\alpha,\beta})$ and Lemma~\ref{lemma:clo} implies
\begin{gather*}
 \hk^n_{\alpha,\beta}[u,u]=\hk_{\alpha,\beta}[\widetilde u,\widetilde u]\geq (s_{\inf}-1)\Vert \widetilde u\Vert_{\L\interval}=(s_{\inf}-1)\Vert u\Vert_{\L(x_{-n},x_{n})}.
\end{gather*}
Therefore, $\mu=s_{\inf}-2$ defined in \eqref{mu}
belongs to the resolvent set of the self-adjoint operator $\Hk\aab^n$ and we have
\begin{gather}\label{wiederesti}
1\leq \dist(\mu,\sigma(\Hk\aab^n)),\qquad n\in\N.
\end{gather}
It is also clear from \eqref{first-lambda1}--\eqref{first-lambda2} and \eqref{sinf} that $\mu$ belongs to the resolvent set 
of the operator in \eqref{finitedecoup} and the same estimate holds.

\smallskip   
 
\noindent\textbf{Base case ($n=1$).}  
We consider the bounded 
operator 
$$T_{\be_0}=(\Hk\aab^1 - \mu \Id )^{-1} - \bigl((\Hk^D_{\I_0,\al_0}\oplus \Hk^D_{\I_1,\al_1}) - \mu \Id \bigr)^{-1}$$ 
in $\L(\I_0\cup\I_1)$. 
For $f,g\in\L(\I_0\cup\I_1)$ we set 
$$u=(\Hk\aab^1- \mu \Id )^{-1}f\quad\text{and}\quad v=\bigl((\Hk^D_{\I_0,\al_0}\oplus \Hk^D_{\I_1,\al_1}) - \mu \Id \bigr)^{-1}g.$$ 
Following the arguments that led to \eqref{res-difference0} one verifies in the present situation that 
\begin{gather}
\label{Tb0}
(T_{\be_0}f,g)_{\L(\I_0\cup\I_1)}=u(x_0)\overline{v'(x_0+0)-v'(x_0-0)}.
\end{gather}
Taking into account that $\be_0>0$ (see~\eqref{be:restriction}), using Lemma~\ref{neumannchen} with the constant $C$ from there, and the definition of the form $\hk\aab^1$ we obtain  
\begin{gather*}
\begin{split}
|u(x_0)|^2&\leq  \frac{1}{\beta_0}\bigg[
\hk_{\I_0,\alpha_0}[\u_0,\u_0]+ {C\over d_0^2}\|\u_0\|^2_{\L(\I_0)}+
\hk_{\I_1,\alpha_1}[\u_1,\u_1]+ {C\over d_1^2}\|\u_1\|^2_{\L(\I_1)} 
  +\,\beta_0 |u(x_0)|^2\bigg]\\
  & =\frac{1}{\beta_0}\bigg[
\hk\aab^1[u,u]+ {C\over d_0^2}\|\u_0\|^2_{\L(\I_0)}+ {C\over d_1^2}\|\u_1\|^2_{\L(\I_1)} \bigg]\\
 &=\frac{1}{\beta_0}\bigg[
 (f+\mu u,u)_{\L(\I_0\cup\I_1)}+ {C\over d_0^2}\|\u_0\|^2_{\L(\I_0)}+ {C\over d_1^2}\|\u_1\|^2_{\L(\I_1)} \bigg].
\end{split}
\end{gather*}
Since $\|\u_k\|_{\L(\I_k)}\leq \|u\|_{\L(\I_0\cup \I_1)}\leq\|f\|_{\L(\I_0\cup \I_1)}$ by \eqref{wiederesti} for $k=0,1$, we conclude 
\begin{gather}\label{u-est0}
|u(x_0)|^2 \leq \frac{C_0}{\beta_0} \|f\|^2_{\L(\I_0\cup \I_1)},
\end{gather}
where the constant $C_0$ depends on $d_0,d_1,\mu$, and $C$ from Lemma~\ref{neumannchen}, but is independent of  $\al_0,\al_{1},\be_0$.
As in the proof of Lemma~\ref{lemma:gamma0} (cf.~\eqref{vk-0}) one obtains the estimate 
\begin{gather*}
|v'(x_0-0)|^2\leq 
C''\left( 
d_0^{-1}  \hk_{\mathcal{I}_0,\alpha_0}[\mathbf{v}_0,\mathbf{v}_0]+ 
 \|\mathbf{H}^D_{\mathcal{I}_0,\alpha_0} \mathbf{v}_0\|^2_{\L(\mathcal{I}_0)}+d_0^{-3}\|\mathbf{v}_0\|^2_{\L(\mathcal{I}_0)}\right)
\end{gather*}
with the constant $C''$ being independent of  $\al_0$.
Using this estimate, and also taking into account that 
$\mathbf{H}^D_{\mathcal I_0,\alpha_0} \mathbf{v}_0 = \mathbf{g}_0 + \mu  \mathbf{v}_0$,
$\mathfrak{h}_{\mathcal{I}_0,\alpha_0}[\mathbf{v}_0,\mathbf{v}_0]=(\mathbf{g}_0+\mu \mathbf{v}_0,\mathbf{v}_0)_{\L(\mathcal{I}_0)},$
and $$\|\mathbf{v}_0\|_{\L(\mathcal{I}_0)}\leq \|\mathbf{g}_0\|_{\L(\mathcal{I}_0)}\leq \|g\|_{\L(\mathcal{I}_0\cup \mathcal{I}_1)},$$ 
we conclude
\begin{gather}
\label{v-est0-}
|v'(x_0-0)|^2\leq C_0^- \|g\|_{\L(\mathcal{I}_0\cup \mathcal{I}_1)}^2,
\end{gather}
and similarly 
\begin{gather}
\label{v-est0+}
|v'(x_0+0)|^2\leq C_0^+ \|g\|_{\L(\mathcal{I}_0\cup \mathcal{I}_1)}^2,
\end{gather}
where the constants $C_0^-$ and $C_0^+$ depend, respectively, on $d_0$ and $d_1$, but are independent of 
$\alpha_0$, $\alpha_1$ and $\beta_0$.
Using \eqref{u-est0}--\eqref{v-est0+} we conclude from \eqref{Tb0} that
\begin{gather}\label{norm}
\|(\Hk\aab^1 - \mu \Id )^{-1} - ((\Hk^D_{\I_0,\al_0}\oplus \Hk^D_{\I_1,\al_1}) - \mu \Id )^{-1}\|\leq \widetilde C_0\beta_0^{-1/2},
\end{gather}
where again the constant $\widetilde C_0>0$ is independent of  $\al_0,\al_{1},\be_0$.
It follows from \eqref{norm} that for each  $j\in\N$ 
\begin{gather}\label{spec-conv}
\sup_{\al_0,\al_1}|\lambda_j(\Hk\aab^1)-\lambda_j(\Hk^D_{\I_0,\al_0}\oplus \Hk^D_{\I_1,\al_1}) |\to 0\text{ as }\beta_0\to\infty.
\end{gather}
Recall, that 
\begin{gather*}
\begin{array}{ll}
\lambda_1(\Hk^D_{\I_0,\al_0})= s_0,& \lambda_1(\Hk^D_{\I_1,\al_1})\in \overline{B_{\delta_1/2}(s_1)},\\ 
\ds \lambda_2(\Hk^D_{\I_0,\al_0})=\left(2\pi\over  d_0\right)^2,& \ds\lambda_2(\Hk^D_{\I_1,\al_1})=\left(2\pi\over  d_1\right)^2;
\end{array}
\end{gather*}
cf. \eqref{first-lambda1}--\eqref{first-lambda2}. 
In particular (cf.~Proposition~\ref{prop:HD:2} and \eqref{alpm}), we get
\begin{gather}\label{recall+}
\lambda_1(\Hk^D_{\I_1,\al_1^-})= s_1-{1\over 2}\delta_1,
\quad
\lambda_1(\Hk^D_{\I_1,\al_1^+})= s_1+{1\over 2}\delta_1.
\end{gather}
Combining  \eqref{spec-conv}--\eqref{recall+},
and taking into account that $ { B_{\delta_0}(s_0)}\cap {B_{\delta_1}(s_1)}=\emptyset$
(since 
$B_{\delta_1}(s_1)\subset (T_1,T_2)\setminus \overline{\opset}$ and
$B_{\delta_0}(s_0)\subset \opset$; cf. \eqref{delta1}, \eqref{delta3}, \eqref{wlog})
we conclude that 
there exists a positive $\beta'_0\ge \be_0^{\inf}$ such that for any  $\beta_0\in [\beta'_0,\infty)$ one has 
\begin{itemize}

\item 
$\lambda_j(\Hk\aab^1) >T_2$ for $j\geq 3$,\smallskip

\item
if $s_0<s_1$, then
$\lambda_1(\Hk\aab^1)\in B_{\delta_0}(s_0)$ and $\lambda_2(\Hk\aab^1)\in B_{\delta_1}(s_1),$
moreover 
$$
\lambda_2(\Hk^1_{\underline\al^{1,1},\beta})<s_1-{1\over 4}\delta_1,
\quad 
s_1+{1\over 4}\delta_1<\lambda_2(\Hk^1_{\overline\al^{1,1},\beta}).
$$ 

\item if $s_1<s_0$, then
$\lambda_1(\Hk\aab^1)\in B_{\delta_1}(s_1)$ and $\lambda_2(\Hk\aab^1)\in B_{\delta_0}(s_0)$,
moreover 
$$
\lambda_1(\Hk^1_{\underline\al^{1,1},\beta})<s_1-{1\over 4}\delta_1,
\quad 
s_1+{1\over 4}\delta_1<\lambda_1(\Hk^1_{\overline\al^{1,1},\beta})
$$ 

\end{itemize}
(recall from \eqref{special:sequence} that $\left((\underline\al^{1,1})_{l}\right)_{l\in\Z}$ and $\left((\overline\al^{1,1})_{l}\right)_{l\in\Z}$ 
are sequences satisfying Hypothesis~\ref{hypo31} with the property
$
(\underline\al^{1,1})_1=\al_1^-,\ 
(\overline \al^{1,1})_1=\al_1^+
$).
Evidently, the above properties yield \eqref{induction}, \eqref{induction:pm} for $n=1$ (recall, that the operator $\Hk\aab^1$ does not 
dependent on $\be_k$ with $k\not= 0$).
\medskip

\noindent\textbf{Induction step ($N\to N+1$)}. Assume that \eqref{induction} and \eqref{induction:pm} hold for some fixed $N\in\N$, that is, 
there exist $\be_k'$, $k=-N+1,\dots,N-1$, such that for $\beta_k'\leq\beta_k$ the spectrum of $\Hk\aab^N$ in $(T_1,T_2)$ 
consists of $2N$ simple eigenvalues which are contained in $B_{\delta_k}(s_k)$, and, moreover, 
the inequalities \eqref{induction:pm} hold with $n=N$.  It is no restriction to assume that $\be_k'$, $k=-N+1,\dots,N-1$, are positive and satisfy  
$\beta_k'\ge\be_k^{\inf}$. Recall that $\Hk\aab^N$ does not dependent on 
$\be_k$ with $|k|> N-1$.

Now let the sequence $\be=(\be_k)_{k\in\Z}$ (which of course satisfies Hypothesis~\ref{hypo41}) be chosen such that
$\beta_k'\leq\beta_k$ holds for $k=-N+1,\dots,N-1$. 
We denote
$\widetilde\I_{N+1}\coloneqq (x_{-N-1},x_{N+1}).$  
For $f,g\in\L(\widetilde\I_{N+1})$ we set 
$$
u=(\Hk\aab^{N+1} - \mu \Id )^{-1}f\quad\text{and}\quad
v=\left(\bigl(\Hk^D_{\I_{-N },\al_{-N }}\oplus\Hk\aab^{N}\oplus \Hk^D_{\I_{N+1},\al_{N+1}}\bigr)- \mu \Id \right)^{-1}g.
$$
Following the arguments that led to \eqref{Tb0} we get the similar equality   
\begin{equation}\label{TbN}
\begin{split}
&\left((\Hk\aab^{N+1} - \mu \Id )^{-1}f - 
\left(\bigl(\Hk^D_{\I_{-N },\al_{-N }}\oplus\Hk\aab^{N}\oplus \Hk^D_{\I_{N+1},\al_{N+1}}\bigr)- \mu \Id \right)^{-1}f,g\right)_{\ds\L(\widetilde\I_{N+1})}\\
&\qquad\qquad=
u(x_N)\overline{v'(x_N+0)-v'(x_N-0)} +
u(x_{-N})\overline{v'(x_{-N}+0)-v'(x_{-N}-0)}.
\end{split}
\end{equation}
Let $C$ be the constant from Lemma~\ref{neumannchen}, and let
$\widehat d_N=\min\left\{d_{-N},d_{-N+1},\dots,d_{N+1}\right\}$. Taking into account that $\be_k>0$ (cf.~\eqref{be:geq:0}), we get
\begin{gather*} 
\begin{split}
|u(x_N)|^2&\leq  \frac{1}{\beta_N}\bigg[\be_N|u(x_N)|^2+
\suml_{k=-N}^{N+1}\left(\hk_{\I_k,\alpha_k}[\u_k,\u_k]+ 
{C\over   d_{k}^2}\|\u_k\|^2_{\L(\I_k)}\right)\bigg]\\
 & \le\frac{1}{\beta_N}\bigg[
\hk\aab^{N+1}[u,u]+   
{C\over  (\widehat d_{N})^2}\|u\|^2_{\L(\widetilde\I_{N+1})} \bigg]\\
 &=\frac{1}{\beta_N}\bigg[
 (f+\mu u,u)_{\L(\widetilde\I_{N+1})}+ {C\over (\widehat d_{N})^2}\|u\|^2_{\L(\widetilde\I_{N+1})} \bigg].
\end{split}
\end{gather*}
Using \eqref{wiederesti} we then arrive at the estimate
\begin{gather*}
|u(x_N)|^2\leq \frac{C_{N}}{\beta_N} \|f\|^2_{\L(\widetilde\I_{N+1})},
\end{gather*}
and, similarly,
\begin{gather*}
|u(x_{-N})|^2\leq \frac{C_{-N}}{\beta_{-N}} \|f\|^2_{\L(\widetilde\I_{N+1})},
\end{gather*}
with the constants $C_{-N},\, C_{N}$ being independent of  
$\al$ and $\be$ (however, they depend on $\widehat d_N$, and   $C$ from Lemma~\ref{neumannchen}).
Similarly to \eqref{v-est0-}--\eqref{v-est0+}, we get the estimates
\begin{gather*}
|v'(x_{-N}- 0)|^2\leq C^-_{-N}\|g\|^2_{\L(\widetilde\I_{N+1})},\\
|v'(x_N+ 0)|^2\leq C^+_{N}\|g\|^2_{\L(\widetilde\I_{N+1})},
\end{gather*}
with the constants $C^-_{-N},\, C^+_{N}$, which depend on $d_{-N}$ and $d_{N+1}$, respectively, but are independent of  $\al$ and $\be$.
Further,
denote $\widetilde v=v\restriction_{(x_{-N},x_N)}$.
As in the proof of Lemma~\ref{lemma:gamma0} 
(cf.~\eqref{vk-0}) one obtains the estimate 
\begin{gather*}
|v'(x_N-0)|^2\leq 
C''\left( 
d_N^{-1}  \hk_{\mathcal{I}_N,\alpha_N}[\mathbf{v}_N,\mathbf{v}_N]+ 
 \|({\mathbf{H}}^{N}\aab \widetilde v)\restriction_{\I_N}\|^2_{\L(\mathcal{I}_N)}+d_N^{-3}\|\mathbf{v}_N\|^2_{\L(\mathcal{I}_N)}\right)
\end{gather*}
with the constant $C''$ being independent of  $\al$ and $\be$.
Using Lemma~\ref{neumannchen} and taking into account that $\be_k>0$ we can extend the above estimate as follows,
\begin{multline*}
|v'(x_N-0)|^2\leq 
C''\bigg[ 
d_N^{-1}  
\bigg(\hk^{N+1}\aab[v,v]
+{C\over (\widehat d_N)^2}\|v\|^2_{\L(\widetilde\I_{N+1})}\bigg) 
\\
+ \| {\mathbf{H}}^{N+1}\aab v \|^2_{\L(\widetilde\I_{N+1})}+d_N^{-3}\|v\|^2_{\L(\widetilde\I_{N+1})}\bigg],
\end{multline*}
where again $C$ is the constant from Lemma~\ref{neumannchen}.
Using 
\begin{gather*}
{\mathbf{H}}^{N+1}\aab v = g + \mu v,\quad
\hk^{N+1}\aab[v,v]=(g+\mu v,v)_{\L(\widetilde\I_{N+1})},\quad
\|v\|_{\L(\widetilde\I_{N+1})}\leq 
\|g\|_{\L(\widetilde\I_{N+1})},
\end{gather*}
we conclude   that
\begin{gather*}
|v'(x_N-0)|^2\leq C_N^- \|g\|_{\L(x_{-N-1},x_{N+1})}^2
\end{gather*}
and, similarly, 
\begin{gather}
\label{v-estN-4}
|v'(x_{-N}+0)|^2\leq C_{-N}^+ \|g\|_{\L(x_{-N-1},x_{N+1})}^2,
\end{gather}
where the constants $C_N^-$ and $C_{-N}^+$ depend on $d_k$, $k=-N,\dots,N+1$, but are independent of 
the sequences $\alpha$ and $\beta$.
Combining \eqref{TbN}--\eqref{v-estN-4} we arrive at the estimate
\begin{multline}\label{norm:N}
\left\|(\Hk\aab^{N+1} - \mu \Id )^{-1} - 
\left(\bigl(\Hk^D_{\I_{-N },\al_{-N }}\oplus\Hk\aab^{N}\oplus \Hk^D_{\I_{N+1},\al_{N+1}}\bigr)- \mu \Id \right)^{-1}\right\|\\\leq
\widetilde C_N\max\bigl\{\beta^{-1/2}_{-N },\beta^{-1/2}_{N }\bigr\},
\end{multline}
where  $\widetilde C_N$ is independent of the sequences $\al$ and $\be$.
Consequently,  for each  $  j\in\N$
\begin{gather}\label{spec-conv1}
\sup _{\al_{-N },\dots,\al_{N+1}}\bigl|\lambda_j(\Hk\aab^{N+1})-\lambda_j\bigl(\Hk^D_{\I_{-N },\al_{-N }}\oplus\Hk\aab^{N}\oplus \Hk^D_{\I_{N+1},\al_{N+1}}\bigr) \bigr|
\to 0\text{ as }\beta_{-N },\beta_{N }\to\infty.
\end{gather}

By construction the set 
$$\sigma\bigl(\Hk^D_{\I_{-N },\al_{-N }}\oplus
\Hk\aab^N\oplus 
\Hk^D_{\I_{N+1},\al_{N+1}}\bigr)\cap (T_1,T_2)$$
consists of 
$2N+2$ simple eigenvalues (we denote them by $\gamma_{\al,\be;k}$, $k=-N ,\dots,N+1$) such that
\begin{gather}\label{gamma}
\begin{array}{l}
\gamma_{\al,\be;-N }= \lambda_1(\Hk^D_{\I_{-N },\al_{-N }}),\quad 
\gamma_{\al,\be;N+1}= \lambda_1(\Hk^D_{\I_{N+1},\al_{N+1}})\\[2mm] 
\gamma_{\al,\be;k}\in B_{\delta_k}(s_k),\,\,\,k=-N+1,\dots,N,
\end{array}
\end{gather}
where $\gamma_{\al,\be;k}$, $k=-N+1,\dots,N$, are the $2N$ simple eigenvalues of the operator $\Hk\aab^N$ inside $(T_1,T_2)$.
Moreover, we then have 
$$\lambda_1(\Hk^D_{\I_{-N },\al_{-N }})=s_{-N}\quad\text{and} \quad\lambda_1(\Hk^D_{\I_{N+1},\al_{N+1}})\in \overline{B_{\delta_{N+1}/2}(s_{N+1})}$$ (cf.~\eqref{first-lambda1}--\eqref{first-lambda2}).
By induction hypothesis, for $n=N$ one has 
\begin{gather}\label{gamma:pm:1}
\gamma_{\underline\al^{n,k},\be;k}< s_k-{1\over 4}\delta_k,\quad
s_k+{1\over 4}\delta_k< \gamma_{\overline\al^{n,k},\be;k},\quad  k=1,\dots,N. 
\end{gather}
Since the eigenvalues $\gamma_{\al,\be;k}$, $k=1,\dots,N$, of $\Hk\aab^N$ are independent of $\al_l$ with $l>N$
it is clear that
\begin{gather}\label{gamma:pm:2}
\text{\eqref{gamma:pm:1} holds also with $n=N+1$.} 
\end{gather}
Moreover, the property $\gamma_{\al,\be;N+1}= \lambda_1(\Hk^D_{\I_{N+1},\al_{N+1}})$ in \eqref{gamma} shows that
\begin{gather}\label{gamma:pm:3}
\gamma_{\underline\al^{N+1,N+1},\be;N+1}= s_{N+1}-{1\over 2}\delta_{N+1},\quad
\gamma_{\overline\al^{N+1,N+1},\be;{N+1}}=s_{N+1}+{1\over 2}\delta_{N+1}.
\end{gather}
Hence, using \eqref{spec-conv1}  we conclude from \eqref{gamma}, \eqref{gamma:pm:2}, \eqref{gamma:pm:3} that 
there exist positive $\be_{-N }'\ge \beta_{-N}^{\inf}$ and $\be_{N }'\ge \beta_{N}^{\inf}$  
such that for $\beta_{-N }\in[\be_{-N }',\infty)$ and $\beta_{N }\in [\be_{N }',\infty)$
the operator
$\Hk\aab^{N+1}$ also has precisely  
$2N+2$ simple eigenvalues $\widetilde\gamma_{\al,\be;k}$, $k=-N ,\dots,N+1$, in the interval $(T_1,T_2)$ 
that
satisfy $\widetilde\gamma_{\al,\be;k}\in B_{\delta_k}( s_k)$ as $k=-N ,\dots,N+1$, and moreover 
\begin{gather*}
\widetilde\gamma_{\underline\al^{n,k},\be;k}< s_k-{1\over 4}\delta_k,\quad
s_k+{1\over 4}\delta_k< \widetilde\gamma_{\overline\al^{n,k},\be;k},\quad  k=1,\dots,N+1,\ n=N+1.
\end{gather*}
Consequently, \eqref{induction} and \eqref{induction:pm} hold for $n=N+1$ and $(\be_k)_{k\in\Z}$ satisfying $\be_k\in [\be_k',\infty)$ as $k=-N,\dots,N$.
Note that $\Hk\aab^{N+1}$ is independent of 
$\be_k$ with $|k|> N$.
This completes the induction step and the proof of
Lemma~\ref{lemma:n:disc:0}.
\end{proof}

Recall that the  eigenvalues $s^n_{\al,\be;\, k}$, $k=1,\dots,n$, of the operator 
$\Hk\aab^n$  are independent of $\al_k$
with $k\not\in\{-n+1,\dots,n\}$ and that the entries $\al_k$ with $k=-n+1,\dots, 0$ are 
fixed; cf. Hypothesis~\ref{hypo31}. Therefore, for a fixed sequence $\beta$ the eigenvalues  $s^n_{\al,\be;\, k}$, $k=1,\dots,n$, can be regarded as functions
of  $\al_1, \al_2,  \dots,\al_{n }$. Bearing this in mind, 
for the following considerations we shall denote the eigenvalues $s^n_{\al,\be;\, k}$, $k=1,\dots,n$, by
$$s^\be_k[ \al_1, \al_2,  \dots,\al_{n }];$$ 
of course we assume here that the sequence $(\be_k)_{k\in\Z}$ satisfies $\be_k'\leq\be_k<\infty$. In particular, the property \eqref{induction:pm} now reads as follows:
\begin{gather}
\label{cond:mp12}
\begin{array}{l}
s^\be_k[\al_1^+,\dots,\al_{k-1}^+,\al_k^-,\al_{k+1}^+,\dots,\al_{n }^+]\leq {s_k-{1\over 4}\delta_k},\quad k=1,\dots,n,\\[2mm]
s^\be_k[\al_1^-,\dots,\al_{k-1}^-,\al_k^+,\al_{k+1}^-,\dots,\al_{n }^-]\geq s_k+{1\over 4}\delta_k,\quad k=1,\dots,n.
\end{array}
\end{gather} 
It is easy to see that the function 
\begin{gather}\label{f}
 f:  (  \al_1,  \al_2,\dots, \al_{n })\mapsto \begin{pmatrix} s^\be_1[ \al_1, \al_2,\dots,  \al_{n } ] \\ \vdots \\
                                                                                 s^\be_n[ \al_1, \al_2,\dots,  \al_{n } ]
                                                                                \end{pmatrix}
\end{gather}
is continuous and each coordinate function $f_k(\cdot)=s^\be_k[\cdot]$ is monotonically increasing in each of its arguments.
Indeed, using the same arguments as in the proof of \eqref{norm:N}, we get the following estimate for two sequences $(\al_k)_{k\in\Z}$ and $(\widetilde\al_k)_{k\in\Z}$:
\begin{gather} \label{norm:al}
\left\|(\Hk\aab^n - \mu \Id )^{-1} - (\Hk^n_{\widetilde\al,\beta} - \mu \Id )^{-1}\right\|
\leq
\widetilde{C}\max_{k=1,\dots,n}|\al_k-\widetilde\al_k|,
\end{gather}
where the constant $\widetilde{C}$ is independent of $\al_k$ and $\widetilde\al_k$ (but it depends on $d_k$, $k=-N+1,\dots,N$). Taking into account that $s_k[\al_1,\dots,\al_n]\in B_{\delta_k}(s_k)$, $k=1,\dots,n$, are simple eigenvalues, we conclude from \eqref{norm:al} the continuity of the function $f$. The monotonicity of $f_k$ in each of its arguments follows from the min-max principle (see, e.g., \cite[Section~4.5]{De95}).

The next lemma is an important ingredient for Theorem~\ref{th:nspec}.

\begin{lemma}\label{lemma:exact:n}
Let $\beta=(\be_k)_{k\in\Z}$ be such that $\be'_k\le\be_k$, $k\in\Z$, where $(\be'_k)_{k\in\Z}$ is a sequence as in Lemma~\ref{lemma:n:disc:0}.
Then for $n\in\N$ the entries $\alpha_1,\dots,\alpha_n$ of the sequence $\al=(\alpha_k)_{k\in \Z}$ 
(satisfying Hypothesis~\ref{hypo31}) can be chosen such that 
\begin{gather*}
s^\be_k[ \al_1, \al_2,  \dots,\al_{n }]=s_k,\qquad k=1,\dots,n.
\end{gather*}
\end{lemma}

The proof of the above lemma is based on the following multi-dimensional version
of the intermediate value theorem, which was established in \cite[Lemma~3.5]{HKP97}.

\begin{lemma}\label{lemma-hempel}
Let $\mathcal{D}=\Pi_{k=1}^n[a_k, b_k]$ with $a_k < b_k$, $k=1,\dots,n$, assume that 
$f:\mathcal{D}\to\R^n$ is continuous and each coordinate function $f_k$ of $f$ is
monotonically increasing in each of its arguments. If
$F_k^-<F_k^+$, $k=1,\dots,n$, where
\begin{gather*}
F_k^-=f_k(b_1,b_2,\dots,b_{k-1},a_k,b_{k+1},\dots,b_n),\quad
F_k^+=f_k(a_1,a_2,\dots,a_{k-1},b_k,a_{k+1},\dots,b_n),
\end{gather*}
then for any $F\in\Pi_{k=1}^n[F_k^-,F_k^+]$
there exists  $x\in\mathcal{D}$ such that $f(x)=F$.
\end{lemma}\smallskip

\begin{proof}[Proof of Lemma~\ref{lemma:exact:n}] We fix $n\in \N$ and set
$\mathcal{D}=\Pi_{k=1}^{n }[\al_k^-,\al_k^+]$; the points 
in $\mathcal{D}$ will be denoted in the form $(  \al_1,  \al_2,\dots, \al_{n })$.
Now consider the function
$
 f:\mathcal{D}\to\mathbb{R}^{n }$ given by \eqref{f}. As noted above, this function
is continuous and each coordinate function $f_k(\cdot) $ is monotonically increasing in each of its arguments.
Moreover, according to \eqref{cond:mp12} we have
\begin{gather*}
 F_k^-\coloneqq 
s^\be_k[\al_1^+,\dots,\al_{k-1}^+,\al_k^-,\al_{k+1}^+,\dots,\al_{n }^+]\leq s_k-\frac{1}{4}\delta_k
\end{gather*}
and
\begin{gather*}
 F_k^+\coloneqq 
s^\be_k[\al_1^-,\dots,\al_{k-1}^-,\al_k^+,\al_{k+1}^-,\dots,\al_{n }^-]\geq s_k+\frac{1}{4}\delta_k
\end{gather*}
for $k=1,\dots,n$, and hence, in particular,  $F_k^-<s_k<F_k^+$ for $k=1,\dots,n$.
Therefore, by Lemma~\ref{lemma-hempel} there exists $(  \al_1,  \al_2,\dots,  \al_{n })\in\mathcal{D}$ such that
$$f(  \al_1,  \al_2,\dots, \al_{n })=(s_1,s_2,\dots,s_{n});$$
this completes the proof of Lemma~\ref{lemma:exact:n}.
\end{proof}

Combining Lemma~\ref{lemma:n:disc:0} and Lemma~\ref{lemma:exact:n} we immediately arrive at the main result of this subsection.

\begin{theorem}
\label{th:nspec} 
Let $\al^n=(\alpha^n_k)_{k\in \Z}$   be a sequence satisfying Hypothesis~\ref{hypo31},
and assume that $\beta=(\be_k)_{k\in\Z}$ is such that $\be'_k\le\be_k$, $k\in\Z$, where $(\be'_k)_{k\in\Z}$ is a sequence as in Lemma~\ref{lemma:n:disc:0}.
Then for $n\in\N$ the entries $\alpha^n_k\in [\al_k^-,\al_k^+]$, $k=1,\dots,n$, can be chosen such that 
\begin{gather}
\label{th:nspec:1} 
\sigma(\Hk^n_{\al^n,\beta})\cap (T_1,T_2)=
\left\{s^n_{\al^n,\be;\, k} : k\in \mathcal{K}^n\cup \{1,\dots,n\}\right\},
\end{gather}
where $s^n_{\al^n,\be;\, k}$ are simple eigenvalues of $\Hk^n_{\al^n,\beta}$ satisfying 
\begin{gather}
\label{th:nspec:2}  
\begin{array}{ll}
s^n_{\al^n,\be;\, k}\in B_{\delta_k}(s_k),&k\in \mathcal{K}^n ,\\[1mm]
s^n_{\al^n,\be;\, k}=s_k,&k\in \{1,\dots,n\}.
\end{array}
\end{gather}
\end{theorem}

\begin{remark}\label{rem:aldependonbe}
We mention that the entries $\al^n_k$, $k=1,\dots,n$, in the sequence $\al^n=(\alpha^n_k)_{k\in \Z}$ chosen above 
depend on the choice of the sequence $\beta$. 
\end{remark}

\subsection{Spectrum of the operator $\H_{\al^n,\be}^n$}

Let $n\in\N$ and let $\Hk\aab^n$ be the self-adjoint operator in $\L(x_{-n},x_{n})$ from the previous subsection.
In this subsection we will investigate the self-adjoint operator 
\begin{gather}\label{oplus}
\H\aab^n=
\left(\bigoplus_{ k\leq -n}\Hk^D_{\I_k,\al_k}\right)
\oplus   \Hk\aab^n \oplus
\left(\bigoplus_{ k\geq n+1}\Hk^D_{\I_k,\al_k}\right)
\end{gather}
acting in 
$$\L\interval=\left(\bigoplus_{ k\leq -n}\L(\I_k)\right)
\oplus \L(x_{-n},x_{n}) \oplus
\left(\bigoplus_{ k\geq n+1}\L(\I_k)\right).$$ 
Informally speaking the operator $\H\aab^n$ is obtained from the decoupled operator $\H\aa$ in Section~\ref{decosec} by adding
$\delta$-couplings of the strength $\be_k$ at \emph{finitely} many points $x_k$, $k=-n+1,\dots, n-1$. It is clear that
$\H\aab^n$ (and $\Hk\aab^n$) is independent of $\be_k$ with $k\not\in\{-n+1,\dots,n-1\}$.

It is convenient to strengthen Hypothesis~\ref{hypo41} and from now on to consider sequences  $\be=(\be_k)_{k\in\Z}$ that satisfy the following
condition.

\begin{hypothesis}\label{hypo61}
The sequence 
$\be=(\beta_k)_{k\in\Z}$ satisfies Hypothesis~\ref{hypo41} and, in addition, it is assumed that
\eqref{rho-cond} holds and $\be'_k\le\be_k$, $k\in\Z$,
where $(\be'_k)_{k\in\Z}$ denotes the sequence in Lemma~\ref{lemma:n:disc:0}.
\end{hypothesis}

The following theorem is a consequence of Theorem~\ref{th:nspec} and the considerations in Section~\ref{sec:deco} 
and Section~\ref{sec:ess}.

\begin{theorem}\label{th:n:essdisc} 
Let $\al^n=(\alpha^n_k)_{k\in \Z}$ and $\be=(\be_k)_{k\in\Z}$ be sequences satisfying Hypothesis~\ref{hypo31} and Hypothesis~\ref{hypo61}, and 
assume that $\al^n=(\alpha^n_k)_{k\in \Z}$ is chosen such that \eqref{th:nspec:1}--\eqref{th:nspec:2} hold and 
\begin{gather}\label{alpha:fix22}
\al^n_k=\FF^D_{d_k}(s_k),\quad k\in\N\setminus\{1,\dots,n\}.
\end{gather}
Then one has 
\begin{gather} 
\label{spec:n:ess}
\sigma_{\ess}(\H_{\al^n,\be}^n)=S_\ess\quad\text{and}\quad
\sigma_{\disc}(\H_{\al^n,\be}^n)\cap (T_1,T_2)=
S_\disc,
\end{gather}
and, moreover, each  $s_k$, $k\in\N$, is a simple eigenvalue of $\H_{\al^n,\be}^n$.
\end{theorem} 

\begin{proof}
It is not difficult to see that the resolvent difference
$$(\H_{\al^n ,\be}^n-\lambda \Id)^{-1}-(\H_{\al^n,\infty}-\lambda \Id)^{-1}$$
is a finite rank operator for any $\lambda\in\rho(\H_{\al^n ,\be}^n)\cap\rho(\H_{\al^n,\infty})$, and hence, in particular, a compact operator in $\L\interval$;
this follows, e.g., by observing that both operators 
$\H_{\al^n,\be}^n$ and $\H_{\al^n,\infty}$ can be viewed as selfadjoint extensions of {the symmetric operator $\H_{\al^n,\be}^n\cap\H_{\al^n,\infty}$}, which has finite defect.
Hence we have
\begin{gather*} 
\sigma_\ess(\H_{\al^n ,\be}^n)=\sigma_\ess(\H_{\al^n,\infty})=S_\ess
\end{gather*}
by Theorem~\ref{bigthm1} and this shows the first assertion in \eqref{spec:n:ess}. 

Now we study the discrete spectrum of the operator $\H_{\al^n ,\be}^n$ in $(T_1,T_2)$. It is clear from \eqref{oplus} that
\begin{gather*}%\label{opsumi}
 \sigma(\H_{\al^n,\be}^n)=\sigma\left(\bigoplus_{ k\leq -n}\Hk^D_{\I_k,\al_k^n}\right)\cup\sigma(\Hk_{\alpha^n,\beta}^n) \cup 
 \sigma\left(\bigoplus_{ k\geq n+1}\Hk^D_{\I_k,\al_k^n}\right).
\end{gather*}
Recall from Proposition~\ref{prop:HD:2} that 
$\lambda_1(\Hk^D_{\I_k,\al^n_k})$ coincides with the unique solution of the equation
$\al^n_k=\FF^D_{d_k}(\lambda)$ on the interval $(0,(2\pi/ d_k)^2)$. Thus, taking into account  \eqref{alpha:fix22} and the second property in \eqref{al:cond}, we arrive at
\begin{gather}\label{disc:cond:3}
\lambda_1(\Hk^D_{\I_k,\al_k^n})=s_k,\qquad k\in\Z\setminus\{1,\dots,n\}.
\end{gather}
 Furthermore, we have $\lambda_j(\Hk^D_{\I_k,\al_k})>T_2$ for $j\geq 2$ by Proposition~\ref{prop:HD:2} and \eqref{d:assump1}.
Observe that for $k\in\Z\setminus\N$ the eigenvalues in \eqref{disc:cond:3} do not contribute to the discrete spectrum of $\H_{\al^n,\be}^n$ in $(T_1,T_2)$
since either $s_k\in\opset\subset S_\ess$ or $s_k\not\in [T_1,T_2]$; cf. \eqref{delta3} and \eqref{Sprop2}. It follows that
\begin{gather*}
 \sigma\left(\bigoplus_{ k\leq -n}\Hk^D_{\I_k,\al_k^n}\right)\cap (T_1,T_2)\subset S_\ess.
\end{gather*}
The above considerations also show 
\begin{gather*}
 \sigma\left(\bigoplus_{ k\geq n+1}\Hk^D_{\I_k,\al_k^n}\right)\cap (T_1,T_2) = \bigl\{s_k:k=n+1,n+2,\dots\bigr\},
\end{gather*}
and all the eigenvalues $s_k$, $k=n+1,n+2,\dots$, are simple by the assumption \eqref{Sprop4}. Finally, by Theorem~\ref{th:nspec} 
the spectrum of $\Hk\aab^n$ in $(T_1,T_2)$ consists of the simple eigenvalues $s_k$, $k=1,\dots,n$, and 
the eigenvalues $s^n_{\al^n,\be;\, k}\in B_{\delta_k}(s_k)$ for $k\in \mathcal{K}^n$. However, it follows from \eqref{delta3} and \eqref{Sprop2} that 
$s^n_{\al^n,\be;\, k}\subset S_\ess$ for $k\in \mathcal{K}^n$. Summing up we conclude 
$$
\sigma_{\disc}(\H_{\al^n,\be}^n)\cap (T_1,T_2)=\bigl\{s_k:k\in\N\bigr\}=S_\disc.
$$
\end{proof}

\section{Discrete spectrum of the operator $\H\aab$\label{sec:disc}}

In this section we complete the proof of our main result Theorem~\ref{th:main:pre}. Recall that the ultimate aim is to show the existence of sequences 
$\al=(\alpha_k)_{k\in \Z}$ and $\be=(\be_k)_{k\in\Z}$ such that \eqref{main-ess} and \eqref{main-disc} hold. 
We have already shown in Theorem~\ref{th:ess} that
the assertion \eqref{main-ess} 
on the essential spectrum of $\H\aab$ holds for all sequences $\al=(\alpha_k)_{k\in \Z}$ and $\be=(\be_k)_{k\in \Z}$ that 
satisfy Hypothesis~\ref{hypo31} and Hypothesis~\ref{hypo41}, respectively. 

From now on  we \emph{fix} a sequence $\be=(\be_k)_{k\in \Z}$ that satisfies Hypothesis~\ref{hypo61}  
(and hence also Hypothesis~\ref{hypo41}). Now we define a sequence $\al=(\alpha_k)_{k\in \Z}$ such that 
Hypothesis~\ref{hypo31} holds and  the statement \eqref{main-disc} on the discrete
spectrum of $\H\aab$ is valid: 
By Theorem~\ref{th:n:essdisc} there exists for each $n\in\N$ a sequence   
$\al^n=(\alpha^n_k)_{k\in \Z}$ such that Hypothesis~\ref{hypo31}, \eqref{alpha:fix22} and \eqref{spec:n:ess} hold, and, 
in particular, we have
\begin{gather}\label{enclo}
\al^-_k\leq \al_k^n\leq \al^+_k,\qquad k\in \Z.
\end{gather}
A usual diagonal process shows that there exist
$n_m\in\N$ with $n_m< n_{m+1}$ and $\lim_{m\to\infty}n_m=\infty$ 
and a sequence $\al=(\al_k)_{k\in\Z}$ such that
\begin{gather}\label{alpha:lim}
\al_k^{n_m} \to \al_k\text{ as } m\to\infty\qquad k\in\Z.
\end{gather}
It also follows that $\alpha_k\in [\alpha_k^-,\alpha_k^+]$ for $k\in\Z$, moreover,
by the second property in \eqref{al:cond} we have $\al_k=\al_k^n=\FF^D_{d_k}(s_k)$ for $k\in\Z\setminus\N$. In other words, the sequence $\al=(\al_k)_{k\in\Z}$ satisfies Hypothesis~\ref{hypo31}. Note also that $\al$ depends on the sequence $\beta$ fixed above; cf.~Remark~\ref{rem:aldependonbe}.

\begin{lemma}\label{lemma:nrc}
For the sequence   $\al$ defined by \eqref{alpha:lim} one has
$$\|(\H_{  \al^{n_m},\be}^{n_m}-\mu \Id)^{-1}-(\H_{   \al ,\be}-\mu \Id)^{-1}\|\to 0,\ m\to\infty,$$
where $\mu$ is defined by \eqref{mu}.
\end{lemma}

Before we prove the above lemma, we observe that
Theorem~\ref{th:ess}, Theorem~\ref{th:n:essdisc} together with Lemma~\ref{lemma:nrc}
immediately imply the main result of this section.

\begin{theorem}\label{th:disc} 
For the sequence   $\al$ defined by \eqref{alpha:lim} one has
\begin{gather*} 
\sigma_{\disc}(\H_{\al ,\be} )\cap (T_1,T_2)=
S_\disc,
\end{gather*}
moreover, each  $s_k$, $k\in\N$, is a simple eigenvalue.
\end{theorem}

\begin{proof}[Proof of Lemma~\ref{lemma:nrc}] 
To simplify the presentation we assume that $n_m=m$ for  $m\in\N$.
Now 
let $f,g\in \L\interval$, and consider $u=(\H\aab - \mu \Id)^{-1}f$ and $v=(\H_{\al^m,\be}^m -\mu \Id)^{-1} g$.
Denote $$T^m =(\H\aab-\mu \Id)^{-1}-(\H_{\al^m,\be}^m -\mu \Id)^{-1}.$$
In the same way as in the proof of Theorem~\ref{th:ess} one computes
\begin{equation}\label{intsqq2}
\begin{split}
&(T^m f,g)_{\L\interval}=(u,\H_{\al^m,\be}^mv)_{\L\interval}-(\H\aab u,v)_{\L\interval}\\
& \qquad=\suml_{k\in\Z} u(y_{k})\overline{\bigl(v'(y_k+0)-v'(y_k-0)\bigr)} - \suml_{k\in\Z} \bigl(u'(y_k+0)-u'(y_k-0)\bigr)\overline{v(y_k)}\\
& \qquad\,\,\,+\suml_{k\in\Z} u(x_{k})\overline{\bigl(v'(x_k+0)-v'(x_k-0)\bigr)} - \suml_{k\in\Z} \bigl(u'(x_k+0)-u'(x_k-0)\bigr)\overline{v(x_k)}.
\end{split}
\end{equation}
Using the boundary conditions for $u\in\dom(\H\aab)$ and $v\in\dom(\H_{\al^m,\be}^m)$ we obtain
\begin{equation}\label{intsqq2+}
 (T^m f,g)_{\L\interval}=\underbrace{\suml_{k\in\N} (\alpha_k^m-\alpha_k) u(y_{k})\overline{v(y_k)}}_{I_1^m\coloneqq } +
 \underbrace{\suml_{\vert k\vert\geq m} u(x_{k})\overline{\bigl(v'(x_k+0)-v'(x_k-0)\bigr)}}_{I_2^m\coloneqq }.
\end{equation}
Indeed, the first two sums on the right hand side in \eqref{intsqq2} reduce to the first sum in \eqref{intsqq2+} since $u$ and $v$ satisfy the $\delta$-jump conditions at 
$y_k$, $k\in\Z$, of the strength $\al_k$ and $\al_k^n$, respectively (recall that $\al_k^n=\al_k=\mathcal{F}^D_{d_k}(s_k)$ as $k\in\Z\setminus\N$).
Also,
since $u$ and $v$ satisfy the same $\delta$-jump conditions at $x_k$, $k=-m+1,\dots,m-1$, and $v(x_k)=0$ for all $k\in\Z\setminus\{-m+1,\dots,m-1\}$
the last two sums on the right hand side in \eqref{intsqq2} reduce to the second sum in \eqref{intsqq2+}.

First we  estimate the term $I_1^m$. Fix $\eps>0$.
It is clear that
\begin{gather}\label{Tm}
|I_1^m|\le
\left(\suml_{k\in\N}| \al^m_k- \al_k|\cdot|u(y_k)|^2\right)^{1/2}
\left(\suml_{k\in\N}| \al^m_k- \al_k|\cdot|v(y_k)|^2\right)^{1/2}.
\end{gather}
Let $(c_k)_{k\in\N}$ be the sequence from \eqref{a-a}. Since $c_k\to 0$ as $k\to\infty$, there exists $K(\eps)\in\ \N$ such that 
\begin{gather}\label{eps1}
c_k\leq \eps\text{ as }k>K(\eps).
\end{gather}
Moreover, due to \eqref{alpha:lim}, there exists $M(\eps)$ such that
\begin{gather}\label{eps2}
\text{for }1\leq k\leq K(\eps):\ |\al^m_k - \al _k|<\eps d_k\text{ as }m\ge M(\eps).
\end{gather}
Combining \eqref{a-a}, \eqref{enclo}, \eqref{eps1}--\eqref{eps2}  we
obtain for $m\ge M(\eps)$:
\begin{multline}\label{sum-est}
\suml_{k\in\N}| \al^m_k- \al_k|\cdot|u(y_k)|^2\\\leq
\suml_{k=1}^{K(\eps)}| \al^m_k- \al_k| \cdot|u(y_k)|^2+
\suml_{k=K(\eps)+1}^\infty (\al_k^+-\al_k^-)\cdot|u(y_k)|^2 \leq 
\eps\suml_{k\in \N}d_k |u(y_k)|^2.
\end{multline}
In what follows, we denote by $\u_k$ and $\vv_k$, $k\in\Z$, the restrictions of the functions $\u$ and $\vv$ to the 
interval $\mathcal{I}_k$.
Recall that $\widehat{C}$ is a positive constant for which \eqref{alpha-inf} holds. Without loss of generality
we may assume that $\widehat{C}\ge 1$.
One has the following standard Sobolev inequality (see, e.g. \cite[Lemma 1.3.8]{BK13}):  
\begin{gather}
\label{Sobolev}
\forall \u\in \W^{1,2}(a,b):\quad |u(a)|\leq L\|\u'\|_{\L(a,b)}^2 + 2L^{-1}\|\u\|^2_{\L(a,b)},
\end{gather}
where $a<b<\infty$,  $L\in (0,b-a]$.
Applying   \eqref{Sobolev} with $(a,b)=(y_k,x_k)\subset\mathcal{I}_k$, $L={d_k\over 2 \widehat{C} }$, we get 
\begin{gather*}
|u(y_k)|^2\leq {d_k\over  2\widehat{C}}\|\u_k'\|^2_{\L(y_k,x_k)}+{4 \widehat{C}\over d_k}\|\u_k\|^2_{\L(y_k,x_k)}\leq {d_k\over  2\widehat{C}}\|\u_k'\|^2_{\L(\I_k)}+{4 \widehat{C}\over d_k}\|\u_k\|^2_{\L(\I_k)}.
\end{gather*}
It is straight forward to check that the above estimate is equivalent to the estimate
\begin{gather*}
d_k|u(y_k)|^2\leq 
\left(1+{d_k \al^m_k\over 2\widehat{C}}\right)^{-1}\left({d_k^2\over 2\widehat{C}}\hk_{\I_k, \al^m_k}[\u_k,\u_k]
+
4\widehat C \|\u_k\|^2_{\L(\I_k)}\right).
\end{gather*}
Therefore, since $\al^m_k\in [-\widehat{C}d_k^{-1},0)$ by \eqref{alpha-inf}, $d_k<\ell_+-\ell_-$, and 
\begin{gather*}
%\label{hh:ineq}
\hk_{\I_k, \al^m_k}[\u_k,\u_k]\leq \hk_{\I_k, \al^m_k,\be_{k-1},\be_k}[\u_k,\u_k]
\end{gather*} 
(this inequality holds since $\be_k\ge 0$), we obtain 
\begin{gather*}
d_k|u(y_k)|^2\leq 
C_1 \hk_{\I_k, \al^m_k,\be_{k-1},\be_k}[\u_k,\u_k] + C_2\|\u_k \|^2_{\L(\I_k)},
\end{gather*}
where $C_1={(\ell_+-\ell_-)^2\over\widehat{C}}$, $C_2=8\widehat{C}$.
From the above estimate and \eqref{sum-est} we conclude
\begin{align}\notag
\suml_{k\in\N}| \al^m_k- \al_k|\cdot |u(y_k)|^2&\leq
C_1\eps \h\aab[u,u] + C_2\eps\|u \|^2_{\L\interval}\\\label{sum-est-u}
&\leq C\eps\|f\|^2_{\L\interval},\ m\ge M(\eps).
\end{align}
Similarly, 
\begin{gather}\label{sum-est-v}
\suml_{k\in\N}| \al^m_k- \al_k|\cdot|v(y_k)|^2\leq 
C\eps\|g\|^2_{\L\interval},\ m\ge M(\eps).
\end{gather}
Combining \eqref{Tm}, \eqref{sum-est-u}, \eqref{sum-est-v}  we conclude that 
\begin{gather}\label{I1}
\forall\eps>0\quad \exists M(\eps)\in\N:\quad
|I_1^m|\leq C\eps\|f\|_{\L\interval}\|g\|_{\L\interval}\text{ as }m\ge M(\eps).
\end{gather}

It remains to estimate the term $I_2^m$. Recall that $D_k=\min\{d_k,d_{k+1}\},\,k\in\Z.$ 
One has
\begin{gather}
\label{I2m}
|I_2^m|\leq
\left(\suml_{ |k|\geq m}D_k^{-3}|u(x_k)|^2\right)^{1/2}
\left(\suml_{  |k|\geq m }D_k^{3}|v'(x_k+0)-v'(x_k-0)|^2\right)^{1/2}.
\end{gather}
Repeating verbatim the arguments of the proofs of \eqref{gaga1++} and \eqref{gamma0est:final}, we obtain
\begin{gather}\label{u-est}
\suml_{  |k|\geq m}D_k^{-3}|u(x_k)|^2\leq C_m\|f\|^2_{\L\interval},\text{ where }C_m\to 0\text{ for }m\to\infty,
\end{gather}
and 
\begin{gather}\label{v-est}
\suml_{k\in\Z}  
D_k^{3}|v'(x_k+0)-v'(x_k-0)|^2\leq C\|g\|^2_{\L\interval}
\end{gather}
(note that the function $v$ in the estimate \eqref{gamma0est:final} vanishes at  $x_k$ for \emph{all} $k\in\Z$, however this property
is not utilized for the proof of \eqref{gamma0est:final}).
Combining \eqref{u-est} and \eqref{v-est} we arrive at
\begin{gather}\label{I2}
|I_2^m|^2\leq C\,C_m\,\|f\|^2_{\L\interval}\|g\|^2_{\L\interval},\text{ where }C_m\to 0\text{ for }m\to\infty.
\end{gather}

The statement of the lemma follows immediately from   \eqref{intsqq2+}, \eqref{I1}, 
\eqref{I2}.
\end{proof}

\appendix

\section{}

For the convenience of the reader we briefly discuss in this appendix infinite orthogonal sums of densely defined 
closed uniformly semibounded forms in Hilbert spaces and the associated self-adjoint operators.

\subsection{Direct sums of Hilbert spaces}
\label{A1}

Let $(\HSk_k)_{k\in \Z}$  be a family of Hilbert spaces and let 
$\prod_{k\in\Z} \HSk_k$ be the Cartesian product of $\HSk_k, k\in\Z$.
The elements of $\prod_{k\in\Z} \HSk_k$ will be denoted by roman letters, while bold letter are used for their components, e.g., 
$u=(\u_k)_{k\in\Z}$, $\u_k\in\HSk_k$. 
The direct sum of  $\HSk_k$, 
$$\HS=\bigoplus_{k\in\Z}\HSk_k,$$
consists of all $u=(\u_k)_{k\in\Z}\in \prod\limits_{k\in\Z} \HSk_k$
such that
\begin{gather}
\label{sum-norm}
\|u\|_{\HS}^2=\suml_{k\in \Z} \|\u_k\|^2_{\HSk_k}<\infty.
\end{gather}
Due to \eqref{sum-norm} one can introduce a scalar product 
on $\HS$ by
\begin{gather}\label{sp}
(u,v)_{\HS}=\suml_{k\in \Z}(u_k,v_k)_{\HSk_k}.
\end{gather}
It then turns out that $V$ is a Hilbert space; cf. \cite[Chapter~1.6,\, Theorem~6.2]{Con85}.

\begin{proposition} 
\label{aprop1}
The space $\HS$ equipped with the scalar product \eqref{sp} is a Hilbert space.
\end{proposition}

\subsection{Direct sums of non-negative forms and associated operators}
\label{A2}

Let $(\HSk_k)_{k\in \Z}$  be a family of Hilbert spaces and let $(\hk_k)_{k\in \Z}$ 
be a family of closed non-negative densely defined sesquilinear forms (for each $k\in\Z$ the form $\hk_k$ acts in the space $\HSk_k$).
By the first representation theorem \cite[$\S$-VI. Theorem~2.1]{K66} there exists a unique 
self-adjoint operator $\Hk_k$ being associated with the form $\hk_k$, i.e. $\dom(\Hk_k)\subset\dom(\hk_k)$ and 
$$\hk_k[\u,\vv ]=(\Hk_k \u,\vv )_{\HSk_k},\qquad \u\in\dom(\Hk_k),\, \vv \in\dom(\hk_k).$$

In the space $\HS$ we define the form $\h$ by
\begin{equation*}
\begin{split}
\h[u,v]&=\suml_{k\in \Z}\hk_k[\u_k,\vv _k],\\
\dom(\h)&=\left\{u=(\u_k)_{k\in\Z}\in\HS: \u_k\in\dom(\hk_k),\ \suml_{k\in \Z}\hk_k[\u_k,\u_k]<\infty\right\}.
\end{split}
\end{equation*}
The form $\h$ is refered to as the direct sum of the forms $\hk_k$; we also use the notation
$$\h=\bigoplus_{k\in\Z}\hk_k.$$

\begin{proposition}\label{aprop2}
The form $\h$ is non-negative, densely defined and closed in $V$.
\end{proposition}

\begin{proof}
It is clear that the form $\h$ is non-negative.
In order to prove that $\h$ is densely defined in $V$ fix 
$v=(\vv _k)_{k\in\Z}\in\HS$ and assume that
\begin{gather}
\label{ort}
(u,v)_\HS=0,\qquad u\in \dom(\h).
\end{gather}
For arbitrary $l\in \Z$ and $\w\in\dom(\hk_l)$ we consider
$$
w^l=(\w^l_k)_{k\in\Z}=\begin{cases} \w & \text{if } k=l, \\ 0 & \text{if } k\not=l.\end{cases}
$$
Then $w^l\in\dom(\h)$ and \eqref{ort} holds with $u=w^l$, which implies $(\w,\vv _l)_{\HSk_l}=0$. As the
form $\hk_l$ is densely defined in $\HSk_l$ it follows that $\vv _l=0$. Since $l\in\Z$ is arbitrary we conclude $v=(\vv _l)_{l\in\Z}=0$, which implies that  
$\h$ is densely defined in $\HS$.

Finally, we verify that $\h$ is closed. Let us equip $\dom(\hk_k)$ with the scalar product
\begin{gather}
\label{sp+}
(\u,\vv )_{\dom(\hk_k)}=\hk_k[\u,\vv ]+(u,v)_{\HSk_k},\qquad \u,\vv \in\dom(\hk_k).
\end{gather}
Since $\hk_k$ is closed by assumption $\dom(\hk_k)$ equipped with the scalar product \eqref{sp+} is a Hilbert space. 
On $\dom(\h)$ we consider the scalar product
\begin{gather}
\label{sp++}
(u,v)_{\dom(\h)}:
=\h[u,v]+(u,v)_{\HS}=\suml_{k\in\Z}(\u_k,\vv _k)_{\dom(\hk_k)}
\end{gather}
for $ u=(\u_k)_{k\in\Z}, v=(\vv _k)_{k\in\Z}\in\dom(\h)$.
By Proposition~\ref{aprop1} $\dom(\h)$ together with the scalar product \eqref{sp++} is also a Hilbert space, that is,
the form $\h$ is closed.
\end{proof}

The proposition above implies that there exists a unique self-adjoint and non-negative operator $\H$ 
associated to the form $\h$. We refer to $\H$ as a direct sum of  $\Hk_k$ and use the notation
$$\H=\bigoplus_{k\in\Z}\Hk_k.$$
As a consequence one obtains the following statement.

\begin{proposition}\label{aprop3}
The non-negative self-adjoint operator $\H$ associated to $\h$ in $V$ is given by 
\begin{equation*}
 \begin{split}
  \H u&=(\Hk_k \u_k)_{k\in \Z},\\
  \dom(\H)&=\left\{u=(\u_k)_{k\in\Z}\in V: \u_k\in\dom(\Hk_k),\,\suml_{k\in\Z}\|\Hk_k \u_k\|_{\HSk_k}^2<\infty \right\}.
 \end{split}
\end{equation*}
\end{proposition}

\subsection{Direct sums of uniformly semibounded forms and associated operators}
\label{A3}

Let again $(\HSk_k)_{k\in \Z}$  be a family of Hilbert spaces and 
let $(\hk_k)_{k\in \Z}$  be a family of densely defined semibounded closed forms. 
We assume, in addition, that there is a uniform lower bound 
\begin{gather*}%\label{c-inf}
C_{\inf}=\inf_{k\in \Z}\,\,\, \inf_{\u\in\dom(\hk_k):\ \|\u\|_{\HSk_k}=1} \hk_k[\u,\u]>-\infty,
\end{gather*}
and consider the family of densely defined non-negative closed forms
$$\widetilde\hk_k[\u,\vv ]=\hk_k[\u,\vv ]-C_{\inf}(\u,\vv )_{\HSk_k},\quad \dom(\widetilde\hk_k)=\dom(\hk_k).$$
By Proposition~\ref{aprop2} the form $$\widetilde\h=\bigoplus_{k\in \Z}\widetilde\hk_k$$ 
is non-negative, densely defined and closed in $V=(\HSk_k)_{k\in \Z}$.

Now, we  define the direct sum $\h=\bigoplus_{k\in\Z} \hk_k$ in $V$ by
$$\h[u,v]=\widetilde \h[u,v]+C_{\inf}(u,v)_{\HS},\quad \dom(\h)=\dom(\widetilde\h).$$
It is clear that the form $\h$ is densely defined, semibounded and closed in $V$; moreover 
$\dom(\h)$ consists of all $u=(\u_k)_{k\in\Z}\in \HS$ such that
\begin{gather}\label{domsi2}
\u_k\in\dom(\hk_k)\quad \text{and}\quad \sum_{k\in\Z}|\hk_k[\u_k,\u_k]|<\infty.
\end{gather}
As in the non-negative case the self-adjoint operator $\H$ associated to $\h$ is refered to as the direct sum of the operators $\Hk_k$.
Then one obtains the following variant of Proposition~\ref{aprop3}.

\begin{proposition}\label{aprop3+}
The semibounded self-adjoint operator $\H$ associated to $\h$ in $V$ is given by 
\begin{equation*}
 \begin{split}
  \H u&=(\Hk_k \u_k)_{k\in \Z},\\
  \dom(\H)&=\left\{u=(\u_k)_{k\in\Z}\in V: \u_k\in\dom(\Hk_k),\,\suml_{k\in\Z}\|\Hk_k \u_k\|_{\HSk_k}^2<\infty \right\}.
 \end{split}
\end{equation*}
\end{proposition}

Let us now assume that the spectrum of each semibounded self-adjoint operator $\Hk_k$ is discrete; 
we denote the corresponding eigenvalues (in nondecreasing order with multiplicities taken into accout) 
by $s_{jk}$, $j\in\N$. Furthermore, we introduce the sequence $S=(s_{jk})_{j\in\N,k\in\Z}$.
The next goal is to describe the spectrum of the operator $\H$.

\begin{theorem}
\label{th:A:spec}
Assume that the spectra of all $\Hk_k$ is discrete and let $S=(s_{jk})_{j\in\N,k\in\Z}$ be the set of all eigenvalues. Then the following assertions hold
for the spectrum of the semibounded self-adjoint operator $\H$ in Proposition~\ref{aprop3+}.
\begin{itemize}

\item[{\rm (i)}] $\lambda$ is an eigenvalue of $\H$ if and only if $\lambda\in\sigma(\Hk_k)$ for some $k\in\Z$. More precisely, one has 
\begin{gather}\label{kernis}
\mathrm{ker}(\H-\lambda\Id)=\bigoplus_{k\in\Z}\mathrm{ker}(\H_k-\lambda\Id)
\end{gather}
and, in particular,
\begin{gather}\label{kermult}
\dim\bigl(\mathrm{ker}(\H-\lambda\Id)\bigr)=
\#\left\{(j,k)\in\N\times\Z:\ s_{jk}=\lambda\right\};
\end{gather}

\item[{\rm (ii)}] $\sigma(\H)=\overline{S}$;

\item[{\rm (iii)}] $\sigma_\ess(\H)=\left\{\text{\rm accumulation points of }S\right\}$.

\end{itemize}
\end{theorem}

\begin{proof}
(i) Let $\lambda$ be an eigenvalue of $\H$ and let $u=(\u_k)_{k\in\Z}\in\HS$ be a corresponding eigenfunction. 
Then one has $\Hk_k \u_k=\lambda\u_k$ for all $k\in\Z$ by Proposition~\ref{aprop3+}.  Moreover, 
since $u\not= 0$ there exists $k\in\Z$ such that $\u_k\not=0$. Therefore, $\lambda$ is an eigenvalue of $\Hk_k$.
Conversely, if $\lambda\in\sigma(\Hk_k)$ for some $k\in\Z$  then $\lambda$ is an eigenvalue of $\Hk_k$. If
$\w$ is a corresponding eigenvector then $\lambda$ is an eigenvalue of $\H$ 
$$u=(\u_k)_{k\in\Z}=\begin{cases} \w & \text{if } l=k,\\ 0 & \text{if } l\not=k,\end{cases}$$
is a corresponding eigenvector. This also shows the equality \eqref{kernis} and the last
statement \eqref{kermult} is obvious.

\medskip
\noindent
(ii) The inclusion $\sigma(\H)\supset {S}$ follows from (i). Since $\sigma(\H)$ is closed we conclude  $\sigma(\H)\supset \overline{S}$. 
To prove the reverse inclusion assume that $\lambda\in \R\setminus\overline{S}$. Then there exists $\delta>0$ such that
\begin{gather*}
%\label{delta-dist}
\dist(\lambda,\sigma(\Hk_k))>\delta,\qquad k\in \Z,
\end{gather*}
and, in particular,
$\lambda$ belongs to the resolvent set of each operator $\Hk_k$.
Now pick some $f=(\mathbf{f}_k)_{k\in\N}\in\HS$ and consider $\u_k=(\Hk_k-\lambda\Id)^{-1}\mathbf{f}_k\in\dom(\Hk_k)\subset\dom(\hk_k)$. Then
$\|\u_k\|_{\HSk_k}\leq \delta^{-1} \|\mathbf{f}_k\|_{\HSk_k}$
and for $u=(\u_k)_{k\in\Z}$ one has $\Hk_k \u_k=\mathbf{f}_k + \lambda\u_k$ and 
\begin{equation*}
\begin{split}
\suml_{k\in\Z}|\hk_k[\u_k,\u_k]|&=\suml_{k\in\Z}\left|(\Hk_k\u_k,\u_k)_{\HSk_k}\right| \\
&=\suml_{k\in\Z}\left|(\mathbf{f}_k,\u_k)_{\HSk_k}+\lambda\|\u_k\|^2_{\HSk_k}\right|\\
&\leq \suml_{k\in\Z}  \left(\frac{1}{\delta}\|\mathbf f_k\|^2_{\HSk_k} + \frac{\vert\lambda\vert}{\delta^2}\|\mathbf f_k\|^2_{\HSk_k}\right) \\
&\leq \left(\frac{1}{\delta}+\frac{|\lambda|}{\delta^2}\right)\|f\|^2_{\HS}.
\end{split}
\end{equation*}
Thus $u\in\dom(\h)$; cf. \eqref{domsi2}. Furthermore, for $v=(\vv _k)_{k\in\Z}\in\dom(\h)$ a similar argument shows 
$$\h[u,v]=\suml_{k\in\Z}\hk_k[\u_k,\vv _k]=\suml_{k\in\Z}\left[(\mathbf{f}_k,\vv _k)_{\HSk_k}+\lambda(\u_k,\vv _k)_{\HSk_k}\right]
=(f+\lambda u,v)_{\HS},$$
and we conclude $u\in\dom(\H)$ and $\H u = f+\lambda u$ from the first representation theorem. Consequently, $\text{ran}(\H-\lambda\Id)=V$ and as $\H$ is self-adjoint this shows that 
$\lambda$ is in the resolvent set of $\H$. Therefore we conclude $\sigma(\H)\subset \overline{S}$.

\medskip
\noindent
(iii) Let $\lambda$ be an accumulation point of $S$. Then 
any open neighborhood of $\lambda$ contains infinitely many elements of $S$.
Therefore, either 
\begin{itemize}
\item [(a)] there is a sequence $(\lambda_l)_{l\in\N}$ such that $\lambda_l\in\sigma(\Hk_{k_l})_{l\in\N}$ with
$\lambda_l\not=\lambda$ and $\lambda_{l}\to \lambda$ as $l\to\infty$, or
\item [(b)] there exists an infinite set $K\subset\Z$ such
$\lambda\in\sigma(\Hk_k)$ for $k\in K$.
\end{itemize}
Using (i) we conclude in the case (a) that each punctured neighborhood of $\lambda$ contains an eigenvalue of $\H$, or
in the case (b) $\lambda$ is an eigenvalue of $\sigma(\H)$ with $\dim(\ker(\H-\lambda\Id))=\infty$.
In both situations we have $\lambda\in\sigma_\ess(\H)$.

Conversely, we have $\sigma_\ess(H)=\sigma(\H)\setminus\sigma_\disc(\H)=\overline{S}\setminus\sigma_\disc(\H)$ by (ii).
One concludes from (i) that the set $\sigma_\disc(\H)$ consists of those $\lambda \in S$ which are isolated and satisfy $$\#\left\{(j,k)\in\N\times\Z:\ s_{jk}=\lambda\right\}<\infty.$$
Now, if $\lambda\in\sigma_\ess(H)$ then it follows that $\lambda\in \overline{ S}$ but $\lambda$ is not isolated or $$\#\left\{(j,k)\in\N\times\Z:\ s_{jk}=\lambda\right\}=\infty.$$
In both cases we conclude that $\lambda$ is an accumulation point of $S$.
\end{proof}

\section*{Acknowledgment\label{sec:ack}}

This research was started when A.K. was a postdoctoral researcher at Graz University of Technology; he gratefully acknowledges financial support of the Austrian Science Fund (FWF) through the project M~2310-N32. 
The work of A.K. is also partly supported by the Czech Science Foundation (GA\v{C}R) through the project 21-07129S. 
J.B. gratefully acknowledges financial support by the Austrian Science Fund (FWF): P 33568-N.
    	This publication is based upon work from COST Action CA 18232 
MAT-DYN-NET, supported by COST (European Cooperation in Science and 
Technology), www.cost.eu.


\begin{thebibliography}{99}

\bibitem{ABMN05}
S.~Albeverio, J.~Brasche, M.~Malamud, H.~Neidhardt, 
Inverse spectral theory for symmetric operators with several gaps: scalar-type Weyl functions,
J. Funct. Anal. 228(1) (2005), 144--188. 

\bibitem{ABN98}
S.~Albeverio, J.~Brasche, H.~Neidhardt, 
On inverse spectral theory for self-adjoint extensions: mixed types of spectra,
J. Funct. Anal. 154(1) (1998), 130--173. 

\bibitem{AGHH05}
S.~Albeverio, F.~Gesztesy, R.~H{\o}egh-Krohn, H.~Holden, 
Solvable models in quantum mechanics. Second edition. With an appendix by Pavel Exner, 
AMS Chelsea Publishing, Providence, RI, 2005.

\bibitem{AKM10}
S.~Albeverio, A.~Kostenko, M.~Malamud, 
Spectral theory of semibounded Sturm-Liouville operators with local interactions on a discrete set, 
J. Math. Phys. 51(10) (2010), 102102.

\bibitem{BHSW10} 
J.~Behrndt, S.~Hassi, H.~de~Snoo, R.~Wietsma,
Monotone convergence theorems for semibounded operators and forms with applications,
Proc. Roy. Soc. Edinburgh Sect. A 140(5) (2010), 927--951. 


\bibitem{BK21}
J.~Behrndt, A.~Khrabustovskyi,
Construction of self-adjoint differential operators with prescribed spectral properties,
Math. Nachr. (accepted); arXiv:1911.04781 [math.SP].

\bibitem{BK13}
G.~Berkolaiko, P.~Kuchment,
Introduction to Quantum Graphs, Mathematical Surveys and Monographs 186, AMS Providence, RI, 2013.

\bibitem{Br89}
J.F.~Brasche, Generalized Schr\"odinger operators, an inverse problem in spectral analysis and the Efimov effect, in ``Stochastic Processes, Physics and Geometry'' (S. Albeverio et al., Eds.), pp. 207--245, World Scientific, Singapore, 1989.  

\bibitem{B04} J.F.~Brasche, Spectral theory for self-adjoint extensions. Spectral theory of Schr\"odinger operators, pp. 51--96, Contemp.
Math. 340, Aportaciones Mat., Amer. Math. Soc., Providence, RI, 2004.

\bibitem{BM99}
J.F.~Brasche, M.~Malamud, H.~Neidhardt,
Weyl functions and singular continuous spectra of self-adjoint extensions, Stochastic processes, physics and geometry: new interplays, II (Leipzig, 1999), 75--84, CMS Conf. Proc., 29, Amer. Math. Soc., Providence, RI, 2000. 

\bibitem{BN95}
J.F.~Brasche, H.~Neidhardt, On the absolutely continuous spectrum of self-adjoint extensions, J. Funct. Anal. 131(2) (1995), 364--385. 

\bibitem{BN96}
J.F.~Brasche, H.~Neidhardt, On the singular continuous spectrum of self-adjoint extensions, Math. Zeitschr. 222(4) (1996), 533--542. 

\bibitem{BNW93.1}
J.F.~Brasche, H.~Neidhardt, J.~Weidmann, On the point spectrum of self-adjoint extensions, Math. Zeitschr. 214 (1993), 343--355. 

\bibitem{BNW93.2}
J.F.~Brasche, H. Neidhardt, J.~Weidmann, On the spectra of self-adjoint extensions, Oper. Theory Adv. Appl. 61 (1993), 29--45. 

\bibitem{CdV87}
Y.~Colin de Verdi\'ere, 
Construction de laplaciens dont une partie finie du spectre est donn\'ee,
Ann. Sci. \'Ecole Norm. Sup. (4) 20(4) (1987), 599--615. 

\bibitem{Con85}
J.~Conway,  A Course in Functional Analysis,
Springer-Verlag, New York, 1985.


\bibitem{De95}
E.B.~Davies,
Spectral theory and differential operators, 
Cambridge University Press, Cambridge, 1995.



\bibitem{HKP97}
R.~Hempel, T.~Kriecherbauer, P.~Plankensteiner, 
Discrete and Cantor spectrum for Neumann Laplacians of combs,
Math. Nachr. 188(1) (1997), 141--168. 

\bibitem{HSS91}
R.~Hempel, L.A.~Seco, B.~Simon,
The essential spectrum of Neumann Laplacians on some bounded singular domains,
J. Funct. Anal. 102(2) (1991), 448--483. 

\bibitem{K66}
T.~Kato, 
Perturbation theory for linear operators,
Springer-Verlag, Berlin-Heidelberg-New York, 1966.



\bibitem{KM10}
A.~Kostenko, M.~Malamud,
1-D Schr\"odinger operators with local point interactions on a discrete set, 
J. Differential Equations 249(2) (2010), 253--304.


\bibitem{KM14}
A.~Kostenko, M.~Malamud, 
Spectral theory of semibounded Schr\"odinger operators with $\delta'$-interactions, 
Ann. Henri Poincar\'e 15(3) (2014), 501--541.



\bibitem{P06}
O.~Post, Spectral convergence of quasi-one-dimensional spaces,
Ann. Henri Poincar\'e 7(5) (2006), 933--973.


\bibitem{RS78}
M.~Reed, B.~Simon, 
Methods of Modern Mathematical Physics. IV. Analysis of Operators, 
Academic Press, New York--London, 1978.

\bibitem{Sch12}
K.~Schm\"udgen, 
Unbounded self-adjoint operators on Hilbert space.
Graduate Texts in Mathematics 265, Springer, Dordrecht, 2012.

 
\bibitem{S78}
B.~Simon,
A canonical decomposition for quadratic forms with applications to monotone convergence theorems, 
J. Funct. Anal. 28 (1978), 377--385


\bibitem{S92}
B.~Simon, The Neumann Laplacian of a jelly roll, Proc. Amer. Math. Soc. 114(3) (1992), 783--785. 

\end{thebibliography}
\end{document}